\documentclass[10pt,pdftex]{amsart}

\usepackage{amsmath,array,rotating,color,appendix}
\usepackage{hyperref}
\usepackage[all]{xy}
\usepackage{tocvsec2}
\usepackage{booktabs}
\usepackage{float}

\usepackage[a4paper,margin=2.5cm,bmargin=1.9cm]{geometry}

\newcommand{\Spin}[1]{\ensuremath{\text{\upshape\rmfamily Spin}(#1)}}

\newcommand{\Sp}[1]{\mathrm{Sp}(#1)}
\newcommand{\SO}[1]{\mathrm{SO}(#1)}
\newcommand{\GL}[1]{\mathrm{GL}(#1,\mathbb R)}
\newcommand{\lieso}[1]{\mathop{\mathfrak{so}}(#1)}
\newcommand{\liesp}[1]{\mathop{\mathfrak{sp}}(#1)}
\newcommand{\lieu}[1]{\mathop{\mathfrak{u}}(#1)}
\newcommand{\liespin}[1]{\mathop{\mathfrak{spin}}(#1)}
\newcommand{\U}[1]{\mathrm{U}(#1)}
\newcommand{\Gtwo}{\mathrm{G}_2}
\newcommand{\End}[1]{\mathrm{End}(#1)}
\newcommand{\Cl}[1]{\mathrm{Cl}(#1)}
\newcommand{\vol}{\mathrm{vol}}

\renewcommand{\Im}{\mathop{\mathrm{Im}}}

\newcommand{\Id}{\mathop{\mathrm{Id}}}
\newcommand{\tr}[1]{\mathop{\mathrm{tr}}(#1)}

\newcommand{\ug}{\;\shortstack{{\tiny def}\\=}\;}

\newcommand{\thform}{\ensuremath{\Theta}}
\newcommand{\snform}{\ensuremath{\Phi}}
\newcommand{\form}[1]{\ensuremath{\snform_{#1}}}
\newcommand{\spinform}[1]{\form{\Spin{#1}}}
\newcommand{\kahlerform}{\ensuremath{\omega}}
\newcommand{\quaternionform}{\ensuremath{\Omega}}

\newcommand{\Mathematica}{\textit{Mathematica}}

\newcommand{\shortform}[1]{\ensuremath{{\scriptstyle\boldsymbol{#1}}}}

\newcommand{\st}{\vert}

\newcommand{\I}{\mathcal{I}} 
\newcommand{\J}{J} 

\newcommand{\CC}{\mathbb{C}}
\newcommand{\HH}{\mathbb{H}}
\newcommand{\RR}{\mathbb{R}}
\newcommand{\ZZ}{\mathbb{Z}}

\newcommand{\OO}{\mathbb{O}}

\newcommand{\RH}{R^\HH}
\newcommand{\LH}{L^\HH}
\newcommand{\RO}{R}
\newcommand{\LO}{L}

\newcommand{\CP}[1]{\CC P^{#1}}
\newcommand{\HP}[1]{\HH P^{#1}}
\newcommand{\OP}[1]{\OO P^{#1}}
\newcommand{\OH}[1]{\OO H^{#1}}

\numberwithin{equation}{section}

\theoremstyle{plain}
\newtheorem{te}{Theorem}[section]
\newtheorem*{te*}{Theorem}
\newtheorem{pr}[te]{Proposition}
\newtheorem{co}[te]{Corollary}

\theoremstyle{definition}
\newtheorem{de}[te]{Definition}

\theoremstyle{remark}
\newtheorem{re}[te]{Remark}

\begin{document}

\title[\Spin{9} and almost complex structures on 16-dimensional manifolds]{$\Spin{9}$ and almost complex structures\\ on 16-dimensional manifolds 
}  

\thanks{Both authors were supported by the MIUR under the PRIN Project ``Geometria Differenziale e Analisi Globale''}
\date{\today}

\author{Maurizio Parton}
\address{Universit\`a di Chieti-Pescara\\ Dipartimento di Scienze, viale Pindaro 87, I-65127 Pescara, Italy}
\email{parton@sci.unich.it}
\author{Paolo Piccinni}
\address{Dipartimento di Matematica\\ Sapienza - Universit\`a di Roma \\
Piazzale Aldo Moro 2, I-00185, Roma, Italy 
}
\email{piccinni@mat.uniroma1.it}


\begin{abstract}
For a $\Spin{9}$-structure on a Riemannian manifold $M^{16}$ we write explicitly the matrix $\psi$ of its K\"ahler $2$-forms and the canonical $8$-form $\spinform{9}$. We then prove that $\spinform{9}$ coincides up to a constant with the fourth coefficient of the characteristic polynomial of $\psi$. This is inspired by lower dimensional situations, related to Hopf fibrations and to $\Spin{7}$. As applications, formulas are deduced for Pontrjagin classes and integrals of $\spinform{9}$ and $\spinform{9}^2$ in the special case of holonomy $\Spin{9}$.
\end{abstract}

\maketitle
\tableofcontents

\settocdepth{chapter}
\listoftables
\settocdepth{section}

\section{Introduction}

Although $\Spin{9}$ belongs to M.~Berger's list in his holonomy theorem, it has been known for a long time that the only simply connected complete Riemannian manifolds with holonomy $\Spin{9}$ are the Cayley projective plane $\OP{2}=\frac{\mathrm{F}_4}{\Spin{9}}$ and its dual, the Cayley hyperbolic plane $\OH{2}=\frac{\mathrm{F}_{4(-20)}}{\Spin{9}}$ (cf.~\cite{al}, \cite{br-gr}, as well as \cite[Chapter 10]{besse2}). It is also known that, on the unique irreducible $16$-dimensional $\Spin{9}$-module $\Delta_9$, the space $\Lambda^8$ of exterior $8$-forms contains a $1$-dimensional invariant subspace $\Lambda^8_1$.
Thus, any generator of $\Lambda^8_1$ can be viewed as a canonical $8$-form $\spinform{9}$ on $\RR^{16}$, which is $\Spin{9}$-invariant with respect to the standard $\Spin{9}$-structure.

In the same year $1972$ when the quoted paper \cite{br-gr} by R.~Brown and A.~Gray appeared, Berger published an article \cite{be} on the Riemannian geometry of rank one symmetric spaces, containing the following very simple definition of a
$\Spin{9}$-invariant $8$-form $\spinform{9}$ in $\RR^{16}$:
\begin{equation}\label{ber}
\spinform{9}\ug c\int_{\OP{1}}p_l^*\nu_l\,dl\enspace.
\end{equation}
Here $\nu_l$ is the volume form on the octonionic lines $l\ug\{(x,mx)\}$ or $l\ug\{(0,y)\}$ in $\OO^2\cong\RR^{16}$, $p_l:\OO^2\rightarrow l$ is the projection on $l$, the integral is taken over the ``octonionic projective line'' $\OP{1}=S^8$ of all the $l\subset\OO^2$ and $c$ is a normalizing constant. In the same article, Berger writes a similar definition: $\form{\Sp{n}\cdot\Sp{1}}\ug c\int_{\HP{n-1}}p_l^*\nu_l\,dl$ for a quaternionic $4$-form in $\HH^n \cong \RR^{4n}$. Note that such definitions of $\spinform{9}$ and $\form{\Sp{n}\cdot\Sp{1}}$ arise from distinguished $8$-planes or $4$-planes in the two geometries, appearing thus very much in the spirit of (at the time forthcoming) calibrations. It is also worth reminding that the stabilizers of $\spinform{9}$ in $\GL{16}$ and of $\form{\Sp{n}\cdot\Sp{1}}$ in $\GL{4n}$ are precisely the subgroups $\Spin{9}$ and $\Sp{n}\cdot\Sp{1}$, respectively (cf.~\cite[Pages 168--170]{co} and~\cite[Page 126]{Sa}).

The paper by Brown and Gray contains a different definition of $\spinform{9}$, as a Haar integral over $\Spin{8}$. A natural question is whether an explicit and possibly simple algebraic expression of $\spinform{9}$ can be written in $\RR^{16}$, in parallel with the usual definitions of the $\Gtwo$-invariant $3$-form $\form{\Gtwo}$ on $\RR^7$ or the $\Spin{7}$-invariant $4$-form $\spinform{7}$ on $\RR^8$ (see for example the books~\cite{J1} and~\cite{J2}). 

Indeed, some such algebraic expressions have already been written. Namely, K.~Abe and M.~Matsubara computed $\spinform{9}$ obtaining its $702$ terms from the triality principle of $\Spin{8}$ (see~\cite{Abe} and~\cite{am}, and note that some of the terms have to be corrected~\cite{amerratum}).
More recently, a different approach has been presented by M.~Castrill\'on~L\`opez, P.~Gadea and I.~Mykytyuk in~\cite{cgm}, where a detailed exam is given for the invariance of elements of $\Lambda^8(\RR^{16})$ under the generators of the group $\Spin{9}$. 

A major progress in understanding $\Spin{9}$-structures came in the context of weak holonomies by the work of Th.~Friedrich: in~\cite{fr} and~\cite{fr2} it is observed that the number of possible ``weakened'' holonomies $\Spin{9}$ is $16$, exactly like in the cases of the groups $\U{n}$ and $\Gtwo$, and also that a $\Spin{9}$-structure on $M^{16}$ can be described as a certain vector subbundle $V^9 \subset \End{TM}$. This fact suggests a similarity between $\Spin{9}$ and the quaternionic group $\Sp{n}\cdot\Sp{1}$.

More precisely, a $\Spin{9}$-structure is a rank $9$ real vector bundle $V^9\subset\End{TM}\rightarrow M$, locally spanned by self-dual involutions $\I_{\alpha}$, for $\alpha=1,\dots,9$, such that $\I_{\alpha}\circ\I_{\beta}=-\I_{\beta}\circ\I_{\alpha}$, for $\alpha\neq\beta$ (cf.\ Definition~\ref{def:spin9structure}). From these data, the local almost complex structures 
\begin{equation}\label{eq:alphaijintro}
\J_{\alpha\beta}\ug\I_{\alpha}\circ\I_{\beta}
\end{equation}
are defined on $M^{16}$, and the $9\times 9$ skew-symmetric matrix of their K\"ahler $2$-forms 
\begin{equation}\label{kae}
\psi\ug(\psi_{\alpha\beta})
\end{equation}
is naturally associated with the $\Spin{9}$-structure. The $36$ differential forms $\psi_{\alpha\beta}$, for $\alpha<\beta$, are thus a local system of K\"ahler $2$-forms of the $\Spin{9}$-manifold $(M^{16},V^9)$.

The first result of this paper is the explicit computation of the $702$ terms of $\spinform{9}$, according to the work by Abe and Matsubara, and on the grounds of Berger's definition of $\spinform{9}$. The computation was performed with the help of the software \Mathematica, and the result is shown in Table~\ref{explicitspin9} at page~\pageref{explicitspin9}.

The second result is the following formula for $\spinform{9}$, see Theorem~\ref{teo:main}.

\begin{te*}\label{mainthm}
Let $\spinform{9}=c\int_{\OP{1}}p_l^*\nu_l\,dl$ be the canonical $8$-form in $\RR^{16}$, and choose the constant $c$ in such a way that all its $702$ terms are integers, with no common factors. Then $c=\frac{110880}{\pi^4}$ and
\begin{equation}\label{eq:main}
\spinform{9}=\frac{1}{360}\tau_4(\psi)\enspace,
\end{equation}
where $\tau_4(\psi)$ is the fourth coefficient of the characteristic polynomial of the matrix $\psi$ of K\"ahler $2$-forms.
\end{te*}


Formula~\eqref{eq:main} for $\spinform{9}$ holds more generally for any $16$-dimensional manifold equipped with a $\Spin{9}$-structure. In particular, when the matrix~\eqref{kae} of K\"ahler forms can be interpreted as the matrix of local curvature forms of a linear connection in the real vector bundle $V^9 \rightarrow M^{16}$, then by Chern-Weil theory its second Pontrjagin class $p_2(V)$ is represented, up to a constant, by the closed form $\tau_4(\psi)$. This is certainly the case for a compact Riemannian manifold $M^{16}$ with holonomy $\Spin{9}$, i.e.\ either $\OP{2}$ or any compact quotient of $\OH{2}$. Thus, the third result of this paper is the representation through $\spinform{9}$ of the second Pontrjagin class of $\OP{2}$ or any compact quotient of $\OH{2}$, and a relation of the integrals of $\spinform{9}$ and $\spinform{9}^2$ with the volumes of $\OP{1}$ and $\OP{2}$ respectively, see Corollaries~\ref{hol} and~\ref{hol'}.

It is worth mentioning that our point of view is not strictly related to $\Spin{9}$ as holonomy, but follows the line of non-integrable geometries. For a unified approach to several non-integrable geometries, see the survey~\cite{AgrSLN}.

In this paper we also develop the analogy between $\Spin{9}$-structures on $16$-dimensional manifolds and either almost complex Hermitian structures in dimension $4$ or almost quaternion Hermitian structures in dimension $8$. This is done in Section~\ref{ld}, where this similarity is explained in the framework of what we call \emph{Hopf structure}, arising from the structure of the symmetry group of a Hopf fibration. In particular, in dimension $8$ the structure group $\Sp{1}\cdot\Sp{2}$ is generated by $5$ involutions, inducing $10$ K\"ahler forms $\theta_{\alpha\beta}$, and the left quaternionic $4$-form appears as the second coefficient of the characteristic polynomial of the matrix $(\theta_{\alpha\beta})$, see Proposition~\ref{pr:quaternion}.

In Section~\ref{spin7} we show that $\Spin{7}$ cannot be defined through $7$ involutions, but nevertheless it admits $21$ K\"ahler forms $\varphi_{\alpha\beta}$, and the structure $4$-form $\spinform{7}$ appears as the second coefficient of the characteristic polynomial of the matrix $(\varphi_{\alpha\beta})$, see Proposition~\ref{pr:spin7tau}.

In Section~\ref{sec:spin9} we explicitly compute the $36$ K\"ahler forms $\psi_{\alpha\beta}$ of a $\Spin{9}$-structure, and we prove that in the characteristic polynomial of $(\psi_{\alpha\beta})$ only the fourth coefficient $\tau_4(\psi)$ survives, see Proposition~\ref{pr:charpoly}. Section~\ref{sec:forma_spin9} is then devoted to the computation of Table~\ref{explicitspin9} and finally, in Section~\ref{sec:charpoly}, we prove that $360\spinform{9}=\tau_4(\psi)$ see Theorem~\ref{teo:main}, and we use Chern-Weil theory to obtain a few relations between $\spinform{9}$ and Pontrjagin classes of compact manifolds with holonomy $\Spin{9}$.

The $36$ almost complex structures $\J_{\alpha\beta}$ given in~\eqref{eq:alphaijintro} will be also used in two forthcoming papers, concerning the classical problem of vector fields on spheres of arbitrary dimension \cite{pp}, and the study of $16$-dimensional manifolds equipped with a locally conformal parallel $\Spin{9}$ metric \cite{pp1}.

For the reader's convenience, Table~\ref{synopsis} presents a list of symbols specific to this paper.
%

\begin{table}[H]
\centering
\begin{tabular}{|c|m{.8\textwidth}|}
\hline
\textbf{Symbol} & \multicolumn{1}{p{.8\textwidth}|}{\centering\textbf{Meaning}} \\ \hline
\centering $1,i,j,k,e,f,g,h$ & Units in the octonions $\OO$, with $ie=f,je=g,ke=h$. See Section~\ref{preliminaries}. \\ \hline
\centering $\spinform{7}$ & Structure $4$-form for $\Spin{7}$. Defined by \eqref{eq:spin7}. \\ \hline
\centering $\shortform{\alpha}$ & Boldfaced and scriptsized. Short for $dx_\alpha$, with $x_\alpha$ coordinates in $\RR^8$. The coordinates in $\RR^{16}$ are $(x_1,\dots,x_8,x_1',\dots,x_8')$, and we write also $\shortform{\alpha'}$ as a shortcut for $dx_\alpha'$. The wedge is omitted, so that $\shortform{123'4'}$ means $dx_1\wedge dx_2\wedge dx_3'\wedge dx_4'$. Note that this notation can be mixed with scalars: $-12\shortform{123'4'}$ means then $12$ times $dx_1\wedge dx_2\wedge dx_3'\wedge dx_4'$. \\ \hline
\centering $\I_\alpha$ & Involutions, same symbol with different meanings. They generate the symmetries of the Hopf fibrations $S^3\longrightarrow S^2$, $S^7\longrightarrow S^4$, $S^{15}\longrightarrow S^8$ for $\alpha=1,\dots,3$, $\alpha=1,\dots,5$, $\alpha=1,\dots,9$ respectively. See~\eqref{eq:Ipauli}, \eqref{eq:IH}, \eqref{eq:IO} and~\ref{involutions}. \\ \hline
\centering $\J_{\alpha\beta}$ & The complex structure $\I_\alpha\I_\beta$. For $\alpha=1,\dots,3$ see~\eqref{eq:Jpauli}; for $\alpha=1,\dots,5$ see~\eqref{eq:Jspin51} and~\eqref{eq:Jspin52}; for $\alpha=1,\dots,9$ see~\eqref{eq:J1} and~\eqref{eq:J2}. \\ \hline
\centering $\RH_\alpha,\LH_\alpha$ & Right and left multiplication in $\HH$. Here $\alpha\in\{i,j,k\}$, see~\eqref{eq:right} and~\eqref{eq:left}. \\ \hline
\centering $\RO_\alpha$ & Right multiplication in $\OO$. Here $\alpha\in\{i,j,k,e,f,g,h\}$, see~\eqref{matricesspin7}. \\ \hline
\centering $\theta=(\theta_{\alpha\beta})$ & Matrix of the K\"ahler forms of $\J_{\alpha\beta}$. Defined only in $\dim=8$, thus $\alpha,\beta=1,\dots,5$, see~\eqref{eq:theta1} and~\eqref{eq:theta2}. \\ \hline
\centering $\thform$ & Sum of the squares of $\theta_{\alpha\beta}$. Defined only in $\dim=8$, see~\eqref{Theta}. \\ \hline
\centering $\omega_{\RH_\cdot}$, $\omega_{\LH_\cdot}$ & K\"ahler forms of $\RH_\cdot$, $\LH_\cdot$. Defined only in $\dim=8$, see~\eqref{eq:leftkahlerH} and~\ref{pr:quaternion}. \\ \hline
\centering $\Omega_L$ & Left quaternionic $4$-form on $\HH^2$, see~\ref{pr:quaternion}. \\ \hline
\centering $\phi_\alpha$ & K\"ahler forms of $\RO_\alpha$. Here $\alpha\in\{i,j,k,e,f,g,h\}$. They generate $\Lambda^2_7$ in the decomposition $\Lambda^2\RR^8=\Lambda^2_7\oplus\Lambda^2_{21}$, see~\eqref{phi7}\ \\ \hline
\centering $\phi'_\alpha$, $\phi''_\alpha$, $\phi'''_\alpha$ & K\"ahler forms generating $\Lambda^2_{21}$ in the decomposition $\Lambda^2\RR^8=\Lambda^2_7\oplus\Lambda^2_{21}$. Here $\alpha\in\{i,j,k,e,f,g,h\}$, see~\eqref{phi21} \\ \hline
\centering $\RO_{\alpha\beta}$ & The complex structure $\RO_\alpha\RO_\beta$, where $\alpha,\beta\in\{i,j,k,e,f,g,h\}$, see~\eqref{21}. \\ \hline
\centering $\varphi=(\varphi_{\alpha\beta})$ & Matrix of the K\"ahler forms of $\RO_{\alpha\beta}$, where $\alpha,\beta\in\{i,j,k,e,f,g,h\}$, see~\eqref{eq:varphi7} and~\ref{pr:spin7tau}. \\ \hline
\centering $\psi=(\psi_{\alpha\beta})$ & Matrix of the K\"ahler forms of $\J_{\alpha\beta}$, where $\alpha,\beta\in\{1,\dots,9\}$, see~\eqref{28},~\eqref{8} and~\ref{pr:charpoly}. \\ \hline
\centering $\tau_\alpha(\psi)$ & The coefficients of $\det(tI-\psi)$. Only $\tau_4$ and $\tau_8$ are non-trivial, see~\ref{pr:charpoly}. \\ \hline
\centering $\omega$ & The $2$-form $c\int_{\CP{1}}p_l^*\nu_l\,dl$. With $c=2/\pi$ we have $\omega=\text{K\"ahler form in }\CC^2$, see~\eqref{integralformC}. \\ \hline
\centering $\Omega$ & The $4$-form $c\int_{\HP{1}}p_l^*\nu_l\,dl$. With $c=-120/\pi^2$ we have $\Omega=$ Right quaternion-K\"ahler form in $\HH^2$, see~\eqref{integralformH}. \\ \hline
\centering $\spinform{9}$ & The $8$-form $c\int_{\OP{1}}p_l^*\nu_l\,dl$. The constant $c=110800/\pi^4$ is chosen in such a way that the coefficients of $\spinform{9}$ be coprime integers, see~\eqref{eq:implicitspin9}. \\ \hline
\end{tabular}
\caption[Synoptic table of symbols]{Synoptic table of symbols specific to this paper.}\label{synopsis}
\end{table}

\section{Preliminaries and notations}\label{preliminaries}

In this section we state some standard facts and notations on octonions, which will be used throughout all the computations in this paper. For details, the reader is referred for instance to~\cite{ha-la}, which is consistent with our notation.

We will denote by $i,j,k$ the units of the quaternions $\HH$. A natural way to look at octonions $\OO$ is then as pairs of quaternions. Accordingly, the multiplication between $x,x'\in\OO$ is defined by writing
\[
x=h_1+h_2 e\enspace, \qquad x'=h'_1+h'_2e\enspace,
\]
and their product as
\begin{equation}\label{oct}
xx'=(h_1h'_1-\overline h'_2 h_2)+(h_2\overline h'_1+h'_2 h_1)e\enspace,
\end{equation}
where $\overline h'_1, \overline h'_2$ are the conjugates of quaternions $h'_1, h'_2\in\HH$ (see for instance \cite[page 139]{KoNFD2}).
Note that the identification
\[
x \leftrightarrow (h_1,h_2)
\]
is \emph{not} an isomorphism between $\OO$ and $\HH^2$ as quaternionic vector spaces. This is instead the case for the map
\begin{equation}
(h_1,h_2) \in \HH^2 \rightarrow h_1 +(kh_2\overline k)e \in\OO
\end{equation}
(cf.~\cite[page 5]{br-ha}), useful to compare structures related to quaternions and octonions. We will use this for example to write down Formulas~\eqref{matricesspin7}, concerning the almost complex structures associated with $\Spin{7}$.

Multiplication in $\OO$ is related through Formula~\eqref{oct} with multiplication in $\HH$. For this reason, in this paper we need to distinguish between them, and we will use the symbols $\RH$, $\LH$ for quaternionic multiplication, reserving $\RO$, $\LO$ to the octonion multiplication.

The conjugation in $\OO$ is defined through the one in $\HH$:
\[
\overline{x}\ug\overline h_1-h_2e\enspace,
\]
and allows to write the non-commutativity of $\OO$ as
\[
\overline{x x'}=\overline x'\overline x\enspace.
\]

The non-associativity of $\OO$ gives rise to the associator
\[
[x,y,z]\ug(xy)z-x(yz)\enspace,
\]
alternating and vanishing whenever two of its arguments are either equal or conjugate. The condition $[x,y,z]=0$ for orthonormal bases $\{x,y,z\}$ defines the \emph{associative 3-planes} $\zeta\subset\RR^7=\Im\OO$, also characterized as the ones closed with respect to the \emph{cross-product}
\[
x\times y\ug-\frac{1}{2}(\overline x y-\overline y x)=\Im(\overline y x)\enspace,\qquad \text{for } x, y\in \Im\OO\enspace.
\] 
The Grassmannian of associative $3$-planes in $\Im\OO$ is the quaternion K\"ahler Wolf space $\Gtwo/\SO{4}$.

The \emph{double cross product} on $\RR^8 =\OO$ is defined by
\begin{equation}\label{eq:doublecross}
x\times y\times z\ug\frac{1}{2}(x({\overline y}z)-z({\overline y}x))\enspace,
\end{equation}
or by the simpler expression $x(\overline y z)$ whenever $x,y,z$ are orthogonal.

If $<,>$ denotes the standard scalar product on $\RR^8$, the $4$-form
\begin{equation}\label{eq:spin7}
\spinform{7}(x,y,z,w)\ug<x,y\times z\times w>
\end{equation}
can be written in terms of the canonical basis $\{dx_1,\dots,dx_8\}\subset\Lambda^1\RR^8$ of $1$-forms in $\RR^8$:
\begin{equation}\label{Phi}
\spinform{7} = \shortform{1234}+\shortform{1256}+\shortform{1357}+\shortform{1368}-\shortform{1278}-\shortform{1467}+\shortform{1458}+\star\enspace,
\end{equation}
where $\shortform{\alpha\beta\gamma\delta}$ (smaller size and boldface) denotes $dx_\alpha\wedge dx_\beta\wedge dx_\gamma\wedge dx_\delta$, and $\star$ denotes the Hodge star, with the agreement that $a+\star\ug a+\star a$.


We will use the above notation $\shortform{\alpha\beta\gamma\delta}$ and $a+\star$ throughout all this paper.

Our definition of $\spinform{7}$ follows the choices in \cite[page 120]{ha-la}. Note that other references like \cite{J1} or \cite{J2} use different signs in some of the terms of $\spinform{7}$. The group $\Spin{7}$ can be defined as the subgroup of $\SO{8}$ leaving the $4$-form $\spinform{7}$ invariant. Equivalently (see also Section~\ref{spin7}), $\Spin{7}$ is the subgroup of $\SO{8}$ generated by the right multiplications $\RO_u$, for all imaginary units $u\in S^6\subset\Im\OO$. 

The $4$-form $\spinform{7}$ is self-dual. Indeed, the following decomposition in orthogonal $\Spin{7}$-invariant components applies to the space $\Lambda^4\RR^8=\Lambda^4_+\oplus \Lambda^4_-$ of $4$-forms in $\RR^8$:
\begin{equation}\label{dec7}
\Lambda^4_+ = \Lambda^4_1 \oplus \Lambda^4_7 \oplus \Lambda^4_{27}\enspace, \qquad
\Lambda^4_- = \Lambda^4_{35}\enspace.
\end{equation}
Here $\Lambda^4_\pm$ denote the self-dual and anti-self-dual $4$-forms, $\Lambda^4_l$ a $l$-dimensional vector space and $\Lambda^4_1$ is generated by $\spinform{7}$ (cf.~for example~\cite[page 240]{J2}). Similarly, $2$-forms in $\RR^8$ give rise to the following $\Spin{7}$-invariant orthogonal decomposition:
\begin{equation}\label{dec7'}
\Lambda^2\RR^8 = \Lambda^2_7 \oplus \Lambda^2_{21}\enspace,
\end{equation}
that will be further commented in Section~\ref{spin7}.

According to what we mentioned in the Introduction, we give now the definition of a $\Spin{9}$-structure in the framework of $G$-structures, that we will use in this paper.

\begin{de}\label{def:spin9structure}
A \emph{$\Spin{9}$-structure} on a Riemannian manifold $M^{16}$ is a rank $9$ vector subbundle $V^9\subset\End{TM}$, locally spanned by nine endomorphisms $\I_\alpha$ satisfying the following conditions:
\begin{equation}\label{top}
\I^2_\alpha=\Id\enspace,\qquad\I^*_\alpha=\I_\alpha\enspace,\qquad\I_\alpha\I_\beta=-\I_\beta\I_\alpha\qquad\text{if }\alpha\neq\beta\enspace,
\end{equation}
where $\I^*_\alpha$ denotes the adjoint of $\I_\alpha$.
\end{de}

Observe that Formula~\eqref{top} implies that compositions of $n$ different $\I_\alpha$'s are complex structures if $n\equiv 2,3 \mod 4$, and involutions if $n\equiv 0,1 \mod 4$.

For $M=\RR^{16}$, $\I_1,\dots,\I_9$ are generators of the Clifford algebra $\Cl{9}$, considered as endomorphisms of its $16$-dimensional real representation $\Delta_9 \cong \RR^{16} \cong \OO^2$. Accordingly, unit vectors $v\in S^8\subset\RR^9$ can be seen as symmetric endomorphisms $v:\Delta_9\rightarrow\Delta_9$ via the Clifford multiplication, and these endomorphisms generate $\Spin{9}$.

An explicit way to describe these generators is by writing $v\in S^8\subset\RR\times\OO$ as $r+u$, where $r\in\RR$, $u\in\OO$ and $r^2+u\overline u=1$, and acting on pairs $(x,x')\in\OO^2$ by
\begin{equation}\label{ha}
\left(
\begin{array}{c} 
x \\ 
x'
\end{array}
\right)\longrightarrow
\left(
\begin{array}{cc}
r & \RO_{\overline u} \\
\RO_u & -r
\end{array}
\right)  
\left(
\begin{array}{c}
x \\
x'
\end{array}
\right)\enspace,
\end{equation}
cf.~\cite[page 288]{ha}.

Observe that Formula~\eqref{ha} describes as well a set of generators for other Lie groups, provided that $v$ is taken respectively in $S^2$ and $S^4$, that is to say, provided that $x,x',u$ in~\eqref{ha} are taken respectively in $\CC$ and $\HH$.

\section{Low dimensions}\label{ld}

Formula~\eqref{ha} can be used to define actions of the spheres $S^2$ on $\CC^2$ and $S^4$ on $\HH^2$, by taking $v\in S^2\subset\RR\times\CC$ and $v\in S^4\subset\RR\times\HH$ respectively. This leads to alternative  definitions of a $\U{2}$-structure on $\RR^4$ and of a $\Sp{1}\cdot \Sp{2}$-structure on $\RR^8$, respectively. We briefly describe the analogy with symmetries of the Hopf fibrations presented in~\cite{gl-wa-zi}.

\begin{de} Let $V^3$ be a rank $3$ vector subbundle of the endomorphism bundle $\End{TM}$ on a Riemannian manifold $M^4$. We call $V^3$ a \emph{complex Hopf structure} on $M^4$ if $V^3$ is locally spanned by involutions $\I_1,\I_2,\I_3$ satisfying relations \eqref{top} and related, on open sets covering $M$, by functions giving $\SO{3}$ matrices .
\end{de}

Our terminology is motivated by  the standard choice $M^4=\CC^2$. Here one gets the standard complex Hopf structure from the elements $(r,u)=(0,1), (0,i), (1,0) \in S^2\subset\RR\times\CC$. Their actions on $\CC^2$ according to~\eqref{ha} generate the (identity component of the group of) symmetries of the Hopf fibration $S^3 \longrightarrow S^2$. 

We obtain in this way the Pauli matrices:
\begin{equation}\label{eq:Ipauli}
\I_1=\left(
\begin{array}{rr}
0 & 1 \\
1 & 0
\end{array}\right)\enspace,\qquad
\I_2=\left(
\begin{array}{rr}
0 & -i \\
i & 0
\end{array}\right)\enspace,\qquad
\I_3=\left(
\begin{array}{rr}
1 & 0 \\
0 & -1
\end{array}\right)\enspace,
\end{equation}
belonging to $\U{2}$. The compositions $\J_{\alpha\beta}\ug\I_\alpha\I_\beta$, for $\alpha<\beta$, are given by the complex structures 
\begin{equation}\label{eq:Jpauli}
\J_{12}=\left(
\begin{array}{rr} 
i & 0 \\
0 & -i
\end{array}
\right)\enspace,\qquad
\J_{13}=\left(
\begin{array}{rr}
0 & -1 \\
1 & 0
\end{array}
\right)\enspace,\qquad
\J_{23}=\left(
\begin{array}{rr}
0 & i \\
i & 0
\end{array}
\right)\enspace,
\end{equation}
acting on $\HH\cong\CC^2$ as multiplication on the right by unit quaternions: $\J_{12}=\RH_i,\J_{13}=\RH_j,\J_{23}=\RH_k$. Similarly, multiplication $\LH_i$ on the left by $i$ coincides with $\J_{123}\ug\I_1\I_2\I_3$.

From this, we see that our datum of $V^3\subset\End{T\RR^4}$ on $\RR^4$ gives rise, through the K\"ahler forms of $\J_{12},\J_{13},\J_{23}, \J_{123}$, to the decomposition of $2$-forms in $\RR^4$ as
\[
\Lambda^2\RR^4 \cong \lieso{4}= \lieu{1} \oplus \lieso{3}\oplus \Lambda^2_2,
\]
and the following observation follows.

\begin{pr} The datum of a complex Hopf structure on a Riemannian manifold $M^4$ is equivalent to an almost Hermitian structure, via the isomorphism $\U{1}\cdot\Sp{1}\cong\U{2}$.
\end{pr}

Similarly, Formula \eqref{top} suggests also the following:

\begin{de}  Let $V^5$ be a rank $5$ vector subbundle of the endomorphism bundle $\End{TM}$ on a Riemannian manifold $M^8$. We say that $V^5$ is a \emph{(right) quaternionic Hopf structure} on $M^8$ if $V^5$ is locally spanned by involutions $\I_1,\dots ,\I_5$ satisfying relations \eqref{top} and related, on open sets covering $M$, by functions giving $\SO{5}$ matrices .
\end{de}

Here the terminology comes again from a Hopf fibration. The standard situation is in fact $M^8=\HH^2$ and a basis of $V^5$ is obtained by Formula \eqref{ha}, with now $(r,u)\in S^4\subset\RR\times\HH$ at the five choices $(r,u)=(0,1),(0,i),(0,j),(0,k),(1,0)$. Their action generate $\Sp{2}\cdot\Sp{1}$, the group of symmetries of the Hopf fibration $S^7 \longrightarrow S^4$, defined by looking at $\HH^2$ as a right quaternionic vector space.

One gets in this way the involutions on $\HH^2$:
\begin{equation}\label{eq:IH}
\begin{aligned}
\I_1=\left(
\begin{array}{c|c}
0 & \Id \\ \hline
\Id & 0
\end{array}
\right)\enspace,\qquad
\I_2&=\left(
\begin{array}{c|c}
0 & -\RH_i \\ \hline
\RH_i & 0
\end{array}
\right)\enspace,\qquad &
\I_3&=\left(
\begin{array}{c|c}
0 & -\RH_j \\ \hline
\RH_j & 0
\end{array}
\right)\enspace,
\\
\I_4&=\left(
\begin{array}{c|c}
0 & -\RH_k \\ \hline
\RH_k & 0
\end{array}
\right)\enspace,\qquad &
\I_5&=\left(
\begin{array}{c|c}
\Id & 0 \\ \hline
0 & -\Id
\end{array}
\right)\enspace,
\end{aligned}
\end{equation}
where
\begin{equation}\label{eq:right}
\RH_i=\left(
\begin{array}{rrrr}
0 & -1 & 0 & 0 \\
1 & 0 & 0 & 0 \\
0 & 0 & 0 & 1 \\
0 & 0 & -1 & 0
\end{array}
\right)\enspace,\qquad
\RH_j=\left(
\begin{array}{rrrr}
0 & 0 & -1 & 0 \\
0 & 0 & 0 & -1 \\
1 & 0 & 0 & 0 \\
0 & 1 & 0 & 0
\end{array}
\right)\enspace,\qquad
\RH_k=\left(
\begin{array}{rrrr}
0 & 0 & 0 & -1 \\
0 & 0 & 1 & 0 \\
0 & -1 & 0 & 0 \\
1 & 0 & 0 & 0
\end{array}
\right)\enspace.
\end{equation}

The ten compositions $\J_{\alpha\beta}\ug\I_\alpha\I_\beta$, for $\alpha<\beta$, are thus the following complex structures on $\RR^8$:
\begin{equation}\label{eq:Jspin51}
\begin{aligned}
\J_{12}&=\left(
\begin{array}{c|c}
\RH_i & 0 \\ \hline
0 & -\RH_i
\end{array}
\right)\enspace,\qquad &
\J_{13}&=\left(	
\begin{array}{c|c}
\RH_j & 0 \\ \hline
0 & -\RH_j
\end{array}
\right)\enspace,\qquad &
\J_{14}&=\left(
\begin{array}{c|c}
\RH_k & 0 \\ \hline
0 & -\RH_k
\end{array}
\right)\enspace, \\
\J_{23}&=\left(
\begin{array}{c|c}
\RH_k & 0 \\ \hline
0 & \RH_k
\end{array}
\right)\enspace,\qquad &
\J_{24}&=\left(
\begin{array}{c|c}
-\RH_j & 0 \\ \hline
0 & -\RH_j
\end{array}
\right)\enspace,\qquad &
\J_{34}&=\left(
\begin{array}{c|c}
\RH_i & 0 \\ \hline
0 & \RH_i
\end{array}
\right)\enspace, \\
\end{aligned}
\end{equation}
and
\begin{equation}\label{eq:Jspin52}
\begin{aligned}
\J_{15}=\left(
\begin{array}{c|c}
0 & -\Id \\ \hline
\Id & 0
\end{array}
\right)\enspace,\quad
\J_{25}&=\left(
\begin{array}{c|c}
0 & \RH_i \\ \hline
\RH_i & 0
\end{array}
\right)\enspace,\quad &
\J_{35}&=\left(
\begin{array}{c|c}
0 & \RH_j \\ \hline
\RH_j & 0
\end{array}
\right)\enspace,\quad &
\J_{45}&=\left(
\begin{array}{c|c}
0 & \RH_k \\ \hline
\RH_k & 0
\end{array}
\right)\enspace.
\end{aligned}
\end{equation}

We obtain also ten further complex structures $\J_{\alpha\beta\gamma}\ug\I_\alpha\I_\beta\I_\gamma$, for $\alpha<\beta<\gamma$, that are easily seen to coincide -up to the sign of the permutation
$
\tiny{
\left(
\begin{array}{rrrrr}
1 & 2 & 3 & 4 & 5 \\
\alpha & \beta & \gamma & \delta & \epsilon
\end{array}
\right)
}
$- with the former $\J_{\delta \epsilon}$. Moreover, compositions $\I_\alpha\I_\beta\I_\gamma\I_\delta$ reproduce -up to the negative of the sign of the above permutation- the five involutions $\I_\epsilon$. Recall now that that a $\Sp{1} \cdot \Sp{2}$-structure , i.e. a left quaternion Hermitian structure, on $\RR^8$ is equivalent to decomposing $2$-forms as
\[
\Lambda^2 \RR^8 \cong \lieso{8} = \lieso{2} \oplus \liesp{2} \oplus \Lambda^2_{15},
\]
where $\liesp{2} \cong \lieso{5}$ is generated by the $\J_{\alpha \beta}$ (cf.~\cite[page 125]{Sa}).

Using the notation introduced with Formula~\eqref{Phi}, we can write the K\"ahler forms $\theta_{\alpha\beta}$ of $\J_{\alpha\beta}$ as
\begin{equation}\label{eq:theta1}
\begin{aligned}
\theta_{12} &= -\shortform{12}+\shortform{34}+\shortform{56}-\shortform{78}\enspace, &
\theta_{13} &= -\shortform{13}-\shortform{24}+\shortform{57}+\shortform{68}\enspace, &
\theta_{14} &= -\shortform{14}+\shortform{23}+\shortform{58}-\shortform{67}\enspace,\\
\theta_{23} &= -\shortform{14}+\shortform{23}-\shortform{58}+\shortform{67}\enspace, &
\theta_{24} &= \shortform{13}+\shortform{24}+\shortform{57}+\shortform{68}\enspace, &
\theta_{34} &= -\shortform{12}+\shortform{34}-\shortform{56}+\shortform{78}\enspace,\\
\end{aligned}
\end{equation}
and
\begin{equation}\label{eq:theta2}
\begin{aligned}
\theta_{15} &= -\shortform{15}-\shortform{26}-\shortform{37}-\shortform{48}\enspace, &
\theta_{25} &= -\shortform{16}+\shortform{25}+\shortform{38}-\shortform{47}\enspace, \\
\theta_{35} &= -\shortform{17}-\shortform{28}+\shortform{35}+\shortform{46}\enspace, &
\theta_{45} &= -\shortform{18}+\shortform{27}-\shortform{36}+\shortform{45}\enspace,
\end{aligned}
\end{equation}
so that, if $\theta\ug(\theta_{\alpha\beta})$ and $\tau_2({\theta})$ is the second coefficient of its characteristic polynomial, then a $4$-form 
\begin{equation}\label{Theta}
\begin{split}
\thform &\ug\tau_2(\theta)=\theta^2_{12} + \theta^2_{13} + \dots + \theta^2_{45} \\
&= -12\shortform{1234}-4\shortform{1256}-4\shortform{1357}+4\shortform{1368}-4\shortform{1278}-4\shortform{1467}-4\shortform{1458}+\star
\end{split}
\end{equation}
is defined.

We will need also the fact that the left multiplications by $i,j,k$
\begin{equation}\label{eq:left}
\LH_i=\left(
\begin{array}{rrrr}
0 & -1 & 0 & 0 \\
1 & 0 & 0 & 0 \\
0 & 0 & 0 & -1 \\
0 & 0 & 1 & 0
\end{array}
\right)\enspace,\qquad
\LH_j=\left(
\begin{array}{rrrr}
0 & 0 & -1 & 0 \\
0 & 0 & 0 & 1 \\
1 & 0 & 0 & 0 \\
0 & -1 & 0 & 0
\end{array}
\right)\enspace,\qquad
\LH_k=\left(
\begin{array}{rrrr}
0 & 0 & 0 & -1 \\
0 & 0 & -1 & 0 \\
0 & 1 & 0 & 0 \\
1 & 0 & 0 & 0
\end{array}
\right)
\end{equation}
have K\"ahler forms
\begin{equation}\label{eq:leftkahlerH}
\omega_{\LH_i}=-\shortform{12}-\shortform{34}-\shortform{56}-\shortform{78}\enspace,\qquad
\omega_{\LH_j}=-\shortform{13}+\shortform{24}-\shortform{57}+\shortform{68}\enspace,\qquad
\omega_{\LH_k}=-\shortform{14}-\shortform{23}-\shortform{58}-\shortform{67}\enspace,
\end{equation}
and if 
\[
\Omega_L\ug\omega^2_{\LH_i}+\omega^2_{\LH_j}+\omega^2_{\LH_k}
\]
is the (left) quaternionic $4$-form of $\HH^2\cong\RR^8$, the comparison gives
\begin{equation}\label{proportional}
-2\Omega_L=\thform{}\enspace.
\end{equation}

This proves the following:
\begin{pr}\label{pr:quaternion}
The skew-symmetric matrix $\theta=(\theta_{\alpha\beta})$, whose entries are the
K\"ahler forms of the complex structures $\J_{\alpha\beta}$ on $\RR^8$, allows to construct both the left quaternionic $4$-form $\Omega_L$ and the right K\"ahler $2$-forms $\omega_{\RH_i},\omega_{\RH_j},\omega_{\RH_k}$ as
\[
\Omega_L=-\frac{1}{2}\tau_2(\theta)
\]
and
\[
\omega_{\RH_i}=\theta_{34}\enspace,\qquad\omega_{\RH_j}=-\theta_{24}\enspace,\qquad\omega_{\RH_k}=\theta_{23}\enspace.
\]
\end{pr}

On the other hand, one can easily check that matrices 
$
B=\left(
\begin{array}{c|c}
B' & B'' \\ 
\hline 
B''' & B'''' 
\end{array}
\right)\in\SO{8}
$ that commute with each of the involutions $\I_1,\dots,\I_5$ are the ones satisfying $B''=B'''=0$ and $B'=B''''\in\Sp{1}\subset\SO{4}$. Thus the subgroup preserving each of $\I_1,\dots,\I_5$ is the diagonal $\Sp{1}_\Delta \subset\SO{8}$. The subgroup of $\SO{8}$ we are interested in is indeed the structure group of the quaternionic Hopf structure $V^5$. It consists of matrices $B$ satisfying $B\I_\alpha=\I'_\alpha B$, with $\I_1,\dots,\I_5$ and $\I'_1,\dots,\I'_5$ bases of $V^5$ related by a $\SO{5}$ matrix. Thus this structure group is $\Sp{1}\cdot\Sp{2}$, hence:

\begin{pr} The datum of a (right) quaternionic Hopf structure on a Riemannian manifold $M^8$ is equivalent to a (left) almost quaternion Hermitian structure, i.e.\ to a $\Sp{1}\cdot\Sp{2}$-structure.
\end{pr}

In the above discussion we looked at the standard  $\U{2}$ and $\Sp{1}\cdot\Sp{2}$-structures on $\RR^4$ and $\RR^8$, through the decompositions of $2$-forms
\[
\lieso{4}=\lieu{1}\oplus\lieso{3}\oplus\Lambda^2_2\enspace,\qquad\lieso{8}=\liesp{1}\oplus\liesp{2}\oplus\Lambda^2_{15}
\]
and orthonormal frames in the component $\lieso{3}$ and $\liesp{2}$, respectively. The last components $\Lambda^2_2$ and $\Lambda^2_{15}$ describe all the similar structures on the linear spaces $\RR^4$ and $\RR^8$. Thus, such decompositions give rise to the spaces $\SO{4}/\U{2}$ and $\SO{8}/\Sp{1}\cdot\Sp{2}$, the spaces of all possible structures in the two cases.

Definition~\ref{def:spin9structure} of a $\Spin{9}$-structure on a Riemannian manifold $M^{16}$ corresponds to what is now coherent to call an \emph{octonionic Hopf structure} on $M^{16}$, via Formula~\eqref{top} and choices
\[
(r,u)=(0,1),(0,i),(0,j),\dots,(0,h),(1,0)\in S^8 \subset \RR\times\OO\enspace.
\]
Thus $\Spin{9}$-structures can be viewed as analogues, in dimension $16$, of $\U{2}$ structures in dimension $4$ and of $\Sp{1}\cdot\Sp{2}$ structures in dimension $8$. 

Summarizing:

\begin{co}\label{involutions}
The actions $
\left(
\begin{array}{c} 
x \\ 
x'
\end{array}
\right)\longrightarrow
\left(
\begin{array}{cc}
r & \RO_{\overline u} \\
\RO_u & -r
\end{array}
\right)  
\left(
\begin{array}{c}
x \\
x'
\end{array}
\right)\enspace $, when $u,x,x' \in \CC,\HH,\OO$ (and in any case $r \in \RR$ and $r^2+u\overline u =1$) generate the groups $\U{2}$, $\Sp{1}\cdot\Sp{2}$, $\Spin{9}$ of symmetries
of the Hopf fibrations
\[
S^3 \longrightarrow S^2, \qquad S^7 \longrightarrow S^4, \qquad S^{15} \longrightarrow S^8
\]
(in the first case just the identity component, cf. \cite{gl-wa-zi}). The corresponding $G$-structures on Riemannian manifolds $M^4$, $M^8$, $M^{16}$ can be described through vector subbundles $V \subset \End TM$ of rank $3,5,9$, respectively. Any such $V$ is locally generated by self-dual involutions $\mathcal I_\alpha$ satisfying $\mathcal I_\alpha \mathcal I_\beta = -\mathcal I_\beta \mathcal I_\alpha$ for $\alpha \neq \beta$ and related, on open neighborhoods covering $M$, by functions giving matrices in $\SO{3}$, $\SO{5}$, $\SO{9}$.
\end{co}

\section{The K\"ahler forms of a $\Spin{7}$-structure}\label{spin7}

We saw that $\U{2}$, $\Sp{1}\cdot\Sp{2}$ and $\Spin{9}$ can be described through $3$, $5$ and $9$ involutions satisfying relations~\eqref{top}. We now show that a similar approach cannot be pursued with $\Spin{7}$-structures, that is, $\Spin{7}$ cannot be described by $7$ involutions satisfying relations~\eqref{top}.


\begin{pr}\label{pr:ind}
Let $\I_1,\dots,\I_{2k+1}$ be involutions in $\RR^n$ satisfying~\eqref{top}. The compositions $\J_{\alpha\beta}\ug\I_\alpha\I_\beta$, for $\alpha<\beta$, and $\J_{\alpha\beta\gamma}\ug\I_\alpha\I_\beta\I_\gamma$, for $\alpha<\beta<\gamma$ are linearly independent complex structures on $\RR^n$.
\end{pr}

\begin{proof}
We already observed in Section~\ref{preliminaries} that $\J_{\alpha\beta}$ and $\J_{\alpha\beta\gamma}$ are complex structures. Now observe that \eqref{top}~implies, for any $\alpha=1,\dots,n$, that $\tr{\I^*_\alpha\I_\alpha}=1$, and for $\alpha <\beta$, that $\tr{\I^*_\alpha\I_\beta}=0$. Thus the $\I_\alpha$ are orthonormal and symmetric. By a similar argument, $\tr{\J^*_{\alpha\beta}\J_{\alpha\beta}}=1$ and $\tr{\J^*_{\alpha\beta}\J_{\gamma\delta}}=\tr{\I_\beta\I_\alpha\I_\gamma\I_\delta}=0$ if any of $\gamma,\delta$ equals $\alpha$ or $\beta$. Finally, for $\alpha\neq\gamma$ and $\beta\neq\delta$, note that $\J^*_{\alpha\beta}\J_{\gamma\delta}$ is the composition of the skew-symmetric $\J_{\beta\alpha\gamma}$ and the symmetric $\I_\delta$ and as such its trace is necessarily zero. Similar arguments show that the $\J_{\alpha\beta\gamma}$, for $\alpha<\beta<\gamma$, are orthonormal.
\end{proof}

\begin{co}\label{noreal}
The $\Spin{7}$-structures on $\RR^8$ cannot be defined through $7$ endomorphisms satisfying relations~\eqref{top}.
\end{co}

\begin{proof}
For any choice of $7$ endomorphisms $\{\I_\alpha\}$ in $\RR^8$ satisfying relations~\eqref{top}, the complex structures $\J_{\alpha\beta\gamma}$, for $\alpha<\beta<\gamma$, would give rise to $35$ linearly independent skew-symmetric endomorphisms, by Proposition~\ref{pr:ind}. But this would contradict decomposition~\eqref{dec7'} of $2$-forms in $\RR^8$ under $\Spin{7}$.
\end{proof}

Nevertheless, the right multiplications by $i,j,k,e,f\ug ie,g\ug je,h\ug ke\in\OO$ define $7$ complex structures $\RO_i,\dots,\RO_h$ on $\RR^8$. As mentioned in Section~\ref{preliminaries}, these complex structures lie in $\Spin{7}\subset\SO{8}$.

We will now use Formula~\eqref{oct} to explicitly write the matrix form of $\RO_i,\dots,\RO_h$. If $x=h_1+h_2e\in\OO$, we obtain
\[
\begin{aligned}
xi = h_1i + (-h_2 i)e\enspace,\qquad xj &= h_1j + (-h_2 j)e\enspace,\qquad & xk &= h_1k + (-h_2 k)e\enspace, \\ 
xe &= -h_2 + (h_1)e\enspace,\qquad & xf &= ih_2 + (i h_1)e\enspace,\\
xg &= jh_2 + (jh_1)e\enspace,\qquad & xh &= kh_2 + (kh_1)e\enspace,
\end{aligned}
\]
and thus their matrices read
\begin{equation}\label{matricesspin7}
\begin{aligned}
\RO_i=\left(
\begin{array}{c|c}
\RH_i & 0 \\ \hline
0 & -\RH_i
\end{array}
\right)\enspace,\qquad
\RO_j&=\left(
\begin{array}{c|c}
\RH_j & 0 \\ \hline
0 & -\RH_j
\end{array}
\right)\enspace,\qquad &
\RO_k&=\left(
\begin{array}{c|c}
\RH_k & 0 \\ \hline
0 & -\RH_k
\end{array}
\right)\enspace, \\
\RO_e&=\left(
\begin{array}{c|c}
0 & -\Id \\ \hline
\Id & 0
\end{array}
\right)\enspace,\qquad &
\RO_f&=\left(
\begin{array}{c|c}
0 & \LH_i \\ \hline
\LH_i & 0
\end{array}
\right)\enspace,\\
\RO_g&=\left(
\begin{array}{c|c}
0 & \LH_j \\ \hline
\LH_j & 0
\end{array}
\right)\enspace,\qquad &
\RO_h&=\left(
\begin{array}{c|c}
0 & \LH_k \\ \hline
\LH_k & 0
\end{array}
\right)\enspace.
\end{aligned}
\end{equation}

Correspondingly, their K\"ahler forms
\begin{equation}\label{phi7}
\begin{aligned}
\phi_{i} = -\shortform{12}+\shortform{34}+\shortform{56}-\shortform{78}\enspace,\qquad
\phi_{j} &= -\shortform{13}-\shortform{24}+\shortform{57}+\shortform{68}\enspace,\qquad &
\phi_{k} &= -\shortform{14}+\shortform{23}+\shortform{58}-\shortform{67}\enspace,\qquad \\
\phi_{e} &= -\shortform{15}-\shortform{26}-\shortform{37}-\shortform{48}\enspace,\qquad &
\phi_{f} &= -\shortform{16}+\shortform{25}-\shortform{38}+\shortform{47}\enspace,\qquad \\
\phi_{g} &= -\shortform{17}+\shortform{28}+\shortform{35}-\shortform{46}\enspace,\qquad &
\phi_{h} &= -\shortform{18}-\shortform{27}+\shortform{36}+\shortform{45}
\end{aligned}
\end{equation}
generate the first component $\Lambda^2_7$ in the decomposition \eqref{dec7'}. 

In \cite[page 12]{dhm} it has been observed that $\RR^8$ admits $28=\binom{8}{2}$ linearly independent K\"ahler forms and that they can be defined, up to sign, as the right hand sides of~\eqref{phi7}, corrected either with all plus signs or with an even number of minus signs. Thus, the remaining $21$ K\"ahler forms are generators of the other component $\Lambda^2_{21}$ in~\eqref{dec7'}, that coincides with the Lie algebra $\liespin{7}$. Explicitly, such generators are
\begin{equation}\label{phi21}
\begin{aligned}
\phi'_{i} &= \shortform{12}+\shortform{34}+\shortform{56}+\shortform{78}\enspace,\qquad & \phi''_{i} &= -\shortform{12}-\shortform{34}+\shortform{56}+\shortform{78}\enspace,\qquad & \phi'''_{i} &= -\shortform{12}+\shortform{34}-\shortform{56}+\shortform{78}\enspace, \\
\phi'_{j} &= \shortform{13}+\shortform{24}+\shortform{57}+\shortform{68}\enspace,\qquad & \phi''_{j} &= -\shortform{13}+\shortform{24}-\shortform{57}+\shortform{68}\enspace,\qquad & \phi'''_{j} &= -\shortform{13}+\shortform{24}+\shortform{57}-\shortform{68}\enspace, \\
\phi'_{k} &= \shortform{14}+\shortform{23}+\shortform{58}+\shortform{67}\enspace,\qquad & \phi''_{k} &= -\shortform{14}-\shortform{23}+\shortform{58}+\shortform{67}\enspace,\qquad & \phi'''_{k} &= -\shortform{14}+\shortform{23}-\shortform{58}+\shortform{67}\enspace, \\
\phi'_{e} &= -\shortform{15}-\shortform{26}+\shortform{37}+\shortform{48}\enspace,\qquad & \phi''_{e} &=-\shortform{15}+\shortform{26}-\shortform{37}+\shortform{48}\enspace,\qquad & \phi'''_{e} &= -\shortform{15}+\shortform{26}+\shortform{37}-\shortform{48}\enspace, \\
\phi'_{f} &= \shortform{16}+\shortform{25}+\shortform{38}+\shortform{47}\enspace,\qquad & \phi''_{f} &= -\shortform{16}-\shortform{25}+\shortform{38}+\shortform{47}\enspace,\qquad & \phi'''_{f} &= -\shortform{16}+\shortform{25}+\shortform{38}-\shortform{47}\enspace, \\
\phi'_{g} &= \shortform{17}+\shortform{28}+\shortform{35}+\shortform{46}\enspace,\qquad & \phi''_{g} &= -\shortform{17}-\shortform{28}+\shortform{35}+\shortform{46}\enspace,\qquad & \phi'''_{g} &= -\shortform{17}+\shortform{28}-\shortform{35}+\shortform{46}\enspace,\\
\phi'_{h} &= \shortform{18}+\shortform{27}+\shortform{36}+\shortform{45}\enspace,\qquad & \phi''_{h} &= -\shortform{18}+\shortform{27}-\shortform{36}+\shortform{45}\enspace,\qquad & \phi'''_{h} &= -\shortform{18}+\shortform{27}+\shortform{36}-\shortform{45}\enspace. 
\end{aligned}
\end{equation}

On the other hand, one can write the compositions $\RO_{\alpha\beta}\ug\RO_\alpha\RO_\beta$, for $\alpha,\beta\in\{i,j,k,e,f,g,h\}$:
\begin{equation}\label{21}
\begin{aligned}
\RO_{ij}&=\left(
\begin{array}{c|c}
-\RH_k & 0 \\ \hline
0 & -\RH_k
\end{array}
\right)\enspace,\quad &
\RO_{ik}&=\left(
\begin{array}{c|c}
\RH_j & 0 \\ \hline
0 & \RH_j
\end{array}
\right)\enspace,\quad &
\RO_{ie}&=\left(
\begin{array}{c|c}
0 & -\RH_i \\ \hline
-\RH_i & 0
\end{array}
\right)\enspace, \\ 
\RO_{if}&=\left(
\begin{array}{c|c}
0 & \RH_i\LH_i \\ \hline
-\RH_i\LH_i & 0
\end{array}
\right)\enspace,\quad &
\RO_{ig}&=\left(
\begin{array}{c|c}
0 & \RH_i\LH_j \\ \hline
-\RH_i\LH_j & 0
\end{array}
\right)\enspace,\quad &
\RO_{ih}&=\left(
\begin{array}{c|c}
0 & \RH_i\LH_k \\ \hline
-\RH_i\LH_k & 0
\end{array}
\right)\enspace, \\ 
\RO_{jk}&=\left(
\begin{array}{c|c}
-\RH_i & 0 \\ \hline
0 & -\RH_i
\end{array}
\right)\enspace,\quad &
\RO_{je}&=\left(
\begin{array}{c|c}
0 & -\RH_j \\ \hline
-\RH_j & 0
\end{array}
\right)\enspace,\quad &
\RO_{jf}&=\left(
\begin{array}{c|c}
0 & \RH_j\LH_i \\ \hline
-\RH_j\LH_i & 0
\end{array}
\right)\enspace, \\
\RO_{jg}&=\left(
\begin{array}{c|c}
0 & \RH_j\LH_j \\ \hline
-\RH_j\LH_j & 0
\end{array}
\right)\enspace,\quad &
\RO_{jh}&=\left(
\begin{array}{c|c}
0 & \RH_j\LH_k \\ \hline
-\RH_j\LH_k & 0
\end{array}
\right)\enspace,\quad &
\RO_{ke}&=\left(
\begin{array}{c|c}
0 & -\RH_k \\ \hline
-\RH_k & 0
\end{array}
\right)\enspace, \\
\RO_{kf}&=\left(
\begin{array}{c|c}
0 & \RH_k\LH_i \\ \hline
-\RH_k\LH_i & 0
\end{array}
\right)\enspace,\quad &
\RO_{kg}&=\left(
\begin{array}{c|c}
0 & \RH_k\LH_j \\ \hline
-\RH_k\LH_j & 0
\end{array}
\right)\enspace,\quad &
\RO_{kh}&=\left(
\begin{array}{c|c}
0 & \RH_k\LH_k \\ \hline
-\RH_k\LH_k & 0
\end{array}
\right)\enspace, \\
\RO_{ef}&=\left(
\begin{array}{c|c}
-\LH_i & 0 \\ \hline
0 & \LH_i
\end{array}
\right)\enspace,\quad &
\RO_{eg}&=\left(
\begin{array}{c|c}
-\LH_j & 0 \\ \hline
0 & \LH_j
\end{array}
\right)\enspace,\quad &
\RO_{eh}&=\left(
\begin{array}{c|c}
-\LH_k & 0 \\ \hline
0 & \LH_k
\end{array}
\right)\enspace, \\
\RO_{fg}&=\left(
\begin{array}{c|c}
\LH_k & 0 \\ \hline
0 & \LH_k
\end{array}
\right)\enspace,\quad &
\RO_{fh}&=\left(
\begin{array}{c|c}
-\LH_j & 0 \\ \hline
0 & -\LH_j
\end{array}
\right)\enspace,\quad &
\RO_{gh}&=\left(
\begin{array}{c|c}
\LH_i & 0 \\ \hline
0 & \LH_i
\end{array}
\right)\enspace, \\
\end{aligned}
\end{equation}
where left and right multiplication on $\HH$ are given by Formulas~\eqref{eq:right} and~\eqref{eq:left}, and their compositions are
\begin{small}
\begin{equation*}
\begin{aligned}
\RH_i\LH_i&=\left(
\begin{array}{rrrr}
-1 & 0 & 0 & 0 \\
0 &-1 & 0 & 0 \\
0 & 0 & 1 & 0 \\
0 & 0 & 0 & 1
\end{array}
\right)\thinspace,\enspace &
\RH_i\LH_j&=\left(
\begin{array}{rrrr}
0 & 0 & 0 & -1 \\
0 & 0 & -1 & 0 \\
0 &-1 & 0 & 0 \\
-1 & 0 & 0 & 0
\end{array}
\right)\thinspace,\enspace &
\RH_i\LH_k&=\left(
\begin{array}{rrrr}
0 & 0 & 1 & 0 \\
0 & 0 & 0 & -1 \\
1 & 0 & 0 & 0 \\
0 &-1 & 0 & 0
\end{array}
\right)\thinspace, \\
\RH_j\LH_i&=\left(
\begin{array}{rrrr}
0 & 0 & 0 & -1 \\
0 & 0 & 1 & 0 \\
0 & 1 & 0 & 0 \\
-1 & 0 & 0 & 0
\end{array}
\right)\thinspace,\enspace &
\RH_j\LH_j&=\left(
\begin{array}{rrrr}
1 & 0 & 0 & 0 \\
0 & -1 & 0 & 0 \\
0 & 0 & 1 & 0 \\
0 & 0 & 0 & -1
\end{array}
\right)\thinspace,\enspace &
\RH_j\LH_k&=\left(
\begin{array}{rrrr}
0 & 1 & 0 & 0 \\
1 & 0 & 0 & 0 \\
0 & 0 & 0 & 1 \\
0 & 0 & 1 & 0
\end{array}
\right)\thinspace, \\
\RH_k\LH_i&=\left(
\begin{array}{rrrr}
0 & 0 & -1 & 0 \\
0 & 0 & 0 & -1 \\
-1 & 0 & 0 & 0 \\
0 & -1 & 0 & 0
\end{array}
\right)\thinspace,\enspace &
\RH_k\LH_j&=\left(
\begin{array}{rrrr}
0 & 1 & 0 & 0 \\
1 & 0 & 0 & 0 \\
0 & 0 & 0 & -1 \\
0 & 0 & -1 & 0
\end{array}
\right)\thinspace,\enspace &
\RH_k\LH_k&=\left(
\begin{array}{rrrr}
-1 & 0 & 0 & 0 \\
0 & 1 & 0 & 0 \\
0 & 0 & 1 & 0 \\
0 & 0 & 0 & -1
\end{array}
\right)\thinspace.
\end{aligned}
\end{equation*}
\end{small}

A computation shows then that their K\"ahler forms $\varphi_{\alpha\beta}$ coincide, up to sign, with the forms~\eqref{phi21}. We write explicitly some of them:
\begin{equation}\label{eq:varphi7}
\phi'_{i} = -\varphi_{gh}\enspace,\qquad
\phi''_{i} = -\varphi_{ef}\enspace,\qquad
\phi'''_{i}= -\varphi_{jk}\enspace,\dots,\enspace
\phi'_{h} = -\varphi_{ig}\enspace,\qquad
\phi''_{h}= -\varphi_{ke}\enspace,\qquad
\phi'''_{h}= \varphi_{jf}\enspace. 
\end{equation}

We can now prove that the $8$-form $\spinform{7}$ defined in~\eqref{Phi} can be recovered from any of the two components in the decomposition \eqref{dec7'}.

\begin{pr}\label{pr:spin7tau}
The $7$ K\"ahler forms $\phi_i,\dots,\phi_h$ of the complex structures $\RO_i,\dots,\RO_h$ and the $21$ K\"ahler forms $\varphi_{ij},\dots,\varphi_{gh}$ of $\RO_{\alpha\beta}$, for $\alpha,\beta\in\{i,j,k,e,f,g,h\}$, satisfy
\[
\spinform{7}=-\frac{1}{6}(\phi_{i}^2+\dots+\phi_{h}^2)=\frac{1}{6}(\varphi_{ij}^2+\varphi_{ik}^2+\dots+\varphi_{gh}^2)=\frac{1}{6}\tau_2(\varphi)\enspace,
\]
where $\tau_2(\varphi)$ is the second coefficient of the characteristic polynomial of $\varphi\ug(\varphi_{\alpha\beta})$.
Thus $\spinform{7}$ is, up to a constant, the sum of squares of elements of an orthonormal basis in any of the components of $\Lambda^2\RR^8=\Lambda^2_7\oplus\Lambda^2_{21}$.
\end{pr}

\begin{proof}
A computation shows that
\begin{equation}\label{globalphi}
\begin{split}
&\phantom{=-}{\phi'_i}^2 + {\phi''_i}^2 + {\phi'''_i}^2 +\dots+ {\phi'_h}^2 + {\phi''_h}^2 + {\phi'''_h}^2 \\
&=6(\shortform{1234}+\shortform{5678})-3(\shortform{15}+\shortform{26}+\shortform{37}+\shortform{48})^2
-6(\shortform{1278}-\shortform{1368}+\shortform{1467}+\shortform{2358}-\shortform{2457}+\shortform{3456}) \\
&=-\phi_{i}^2-\dots-\phi_{h}^2\enspace.
\end{split}
\end{equation}
The conclusion then follows by comparing~\eqref{globalphi} with~\eqref{Phi}.
\end{proof}

\begin{re}
We have listed a certain number of complex structures in $\RR^8$. Indeed, a comparison between the two decompositions $\Lambda^2\RR^8=\Lambda^2_{10} \oplus \Lambda^2_{18}$ (under $\Spin{5}$) and $\Lambda^2\RR^8=\Lambda^2_7 \oplus \Lambda^2_{21}$ (under $\Spin{7}$) can be made more precise in terms of the above mentioned generators of the components. The following identities hold between the ten $\Spin{5}$ K\"ahler forms $\theta_{\alpha\beta}$ and some of the $7$ and of the $21$ K\"ahler forms associated with $\Spin{7}$:
\[
\theta_{12} = \phi_i\enspace,\qquad\theta_{13}=\phi_j\enspace,\qquad\theta_{14}=\phi_k\enspace,\qquad\theta_{15}=\phi_e\enspace,\qquad
\]
and
\[
\theta_{23} =\varphi_{ij}\enspace,\qquad\theta_{24}=-\varphi_{ik}\enspace,\qquad\theta_{34}=-\varphi_{jk}\enspace,\qquad\theta_{25}=\varphi_{ie}\enspace,\qquad\theta_{35}=-\varphi_{je}\enspace,\qquad\theta_{45}=-\varphi_{ke}\enspace. 
\]


It follows that the remaining $2$-forms $\phi_\alpha$ and $\varphi_{\alpha \beta}$ can be chosen as generators of the component $\Lambda^2_{18}$, that contains the $\lieso{3}$ spanned by $\varphi_{fg}, \varphi_{fh}, \varphi_{gh}$.\hfill\qed
\end{re}

\begin{re}
By comparing the last two sections, it appears that the behaviour of the representations of $\Spin{5}$ and of $\Spin{7}$ on $\RR^8$ are quite different in terms of the associated almost complex structures. In particular, Corollary~\ref{noreal} states the impossibility of deducing the almost complex structures defined by $\Spin{7}$ from a set of involutions. As we will see in the next section, $\Spin{9}$ is much closer in this respect to $\Spin{5}$ than to $\Spin{7}$. However, most of the formulas written in the present section will be useful to obtain explicitly the K\"ahler forms associated with $\Spin{9}$.\hfill\qed
\end{re}

\section{The K\"ahler forms of a $\Spin{9}$-structure}\label{sec:spin9}

A basis of the standard $\Spin{9}$-structure on $\OO^2 \cong \RR^{16}$ can be written by looking at~\eqref{ha} and at the $9$ vectors $(0,1),(0,i),(0,j),\dots,(0,h),(1,0)\in S^8 \subset \RR\times\OO$. This gives the following symmetric endomorphisms:
\begin{equation}\label{eq:IO}
\begin{aligned}
\I_1&=\left(
\begin{array}{c|c}
0 & \Id \\ \hline
\Id & 0
\end{array}
\right)\enspace,\qquad &
\I_2&=\left(
\begin{array}{c|c}
0 & -\RO_i \\ \hline
\RO_i & 0
\end{array}
\right)\enspace,\qquad &
\I_3&=\left(
\begin{array}{c|c}
0 & -\RO_j \\ \hline
\RO_j & 0
\end{array}
\right)\enspace, \\
\I_4&=\left(
\begin{array}{c|c}
0 & -\RO_k \\ \hline
\RO_k & 0
\end{array}
\right)\enspace,\qquad &
\I_5&=\left(
\begin{array}{c|c}
0 & -\RO_e \\ \hline
\RO_e & 0
\end{array}
\right)\enspace,\qquad &
\I_6&=\left(
\begin{array}{c|c}
0 & -\RO_f\\ \hline
\RO_f & 0
\end{array}
\right)\enspace, \\
\I_7&=\left(
\begin{array}{c|c}
0 & -\RO_g \\ \hline
\RO_g & 0
\end{array}
\right)\enspace,\qquad &
\I_8&=\left(
\begin{array}{c|c}
0 & -\RO_h \\ \hline
\RO_h & 0
\end{array}
\right)\enspace,\qquad &
\I_9&=\left(
\begin{array}{c|c}
\Id & 0 \\ \hline
0 & -\Id
\end{array}
\right)\enspace.
\end{aligned}
\end{equation}

The space $\Lambda^2\RR^{16}$ of $2$-forms in $\RR^{16}$ decomposes under $\Spin{9}$ as
\[
\Lambda^2\RR^{16} = \Lambda^2_{36} \oplus \Lambda^2_{84}
\]
(cf.~\cite[page 146]{fr}), where  $\Lambda^2_{36} \cong \liespin{9}$ and $ \Lambda^2_{84}$ is an orthogonal complement in $\Lambda^2 \cong \lieso{16}$. Explicit generators of both subspaces can be written by looking at the $36$ compositions $\J_{\alpha \beta} \ug \I_\alpha \I_\beta$, for $\alpha <\beta$, and at the $84$ compositions $\J_{\alpha \beta \gamma} \ug \I_\alpha \I_\beta \I_\gamma$, for $\alpha <\beta<\gamma$, all complex structures on $\RR^{16}$. We write now explicitly the matrix forms for $\J_{\alpha\beta}$, and for convenience we split them in two families. The first $28$ complex structures are
\begin{footnotesize}
\begin{equation}\label{eq:J1}
\begin{aligned}
\J_{12}&=\left(
\begin{array}{c|c}
\RO_i & 0 \\ \hline
0 & -\RO_i
\end{array}
\right),\thinspace &
\J_{13}&=\left(
\begin{array}{c|c}
\RO_j & 0 \\ \hline
0 & -\RO_j
\end{array}
\right),\thinspace &
\J_{14}&=\left(
\begin{array}{c|c}
\RO_k & 0 \\ \hline
0 & -\RO_k
\end{array}
\right),\thinspace &
\J_{15}&=\left(
\begin{array}{c|c}
\RO_e & 0 \\ \hline
0 & -\RO_e
\end{array}
\right),
\\
\J_{16}&=\left(
\begin{array}{c|c}
\RO_f & 0 \\ \hline
0 & -\RO_f
\end{array}
\right),\thinspace &
\J_{17}&=\left(
\begin{array}{c|c}
\RO_g & 0 \\ \hline
0 & -\RO_g
\end{array}
\right),\thinspace &
\J_{18}&=\left(
\begin{array}{c|c}
\RO_h & 0 \\ \hline
0 & -\RO_h
\end{array}
\right),\thinspace &
\J_{23}&=\left(
\begin{array}{c|c}
-\RO_{ij} & 0 \\ \hline
0 & -\RO_{ij}
\end{array}
\right),
\\
\J_{24}&=\left(
\begin{array}{c|c}
-\RO_{ik} & 0 \\ \hline
0 & -\RO_{ik}
\end{array}
\right),\thinspace &
\J_{25}&=\left(
\begin{array}{c|c}
-\RO_{ie} & 0 \\ \hline
0 & -\RO_{ie}
\end{array}
\right),\thinspace &
\J_{26}&=\left(
\begin{array}{c|c}
-\RO_{if} & 0 \\ \hline
0 & -\RO_{if}
\end{array}
\right),,\thinspace &
\J_{27}&=\left(
\begin{array}{c|c}
-\RO_{ig} & 0 \\ \hline
0 & -\RO_{ig}
\end{array}
\right),
\\
\J_{28}&=\left(
\begin{array}{c|c}
-\RO_{ih} & 0 \\ \hline
0 & -\RO_{ih}
\end{array}
\right),\thinspace &
\J_{34}&=\left(
\begin{array}{c|c}
-\RO_{jk} & 0 \\ \hline
0 & -\RO_{jk}
\end{array}
\right),\thinspace &
\J_{35}&=\left(
\begin{array}{c|c}
-\RO_{je} & 0 \\ \hline
0 & -\RO_{je}
\end{array}
\right),\thinspace &
\J_{36}&=\left(
\begin{array}{c|c}
-\RO_{jf} & 0 \\ \hline
0 & -\RO_{jf}
\end{array}
\right),
\\
\J_{37}&=\left(
\begin{array}{c|c}
-\RO_{jg} & 0 \\ \hline
0 & -\RO_{jg}
\end{array}
\right),\thinspace &
\J_{38}&=\left(
\begin{array}{c|c}
-\RO_{jh} & 0 \\ \hline
0 & -\RO_{jh}
\end{array}
\right),\thinspace &
\J_{45}&=\left(
\begin{array}{c|c}
-\RO_{ke} & 0 \\ \hline
0 & -\RO_{ke}
\end{array}
\right),\thinspace &
\J_{46}&=\left(
\begin{array}{c|c}
-\RO_{kf} & 0 \\ \hline
0 & -\RO_{kf}
\end{array}
\right),
\\
\J_{47}&=\left(
\begin{array}{c|c}
-\RO_{kg} & 0 \\ \hline
0 & -\RO_{kg}
\end{array}
\right),\thinspace &
\J_{48}&=\left(
\begin{array}{c|c}
-\RO_{kh} & 0 \\ \hline
0 & -\RO_{kh}
\end{array}
\right),\thinspace &
\J_{56}&=\left(
\begin{array}{c|c}
-\RO_{ef} & 0 \\ \hline
0 & -\RO_{ef}
\end{array}
\right),\thinspace &
\J_{57}&=\left(
\begin{array}{c|c}
-\RO_{eg} & 0 \\ \hline
0 & -\RO_{eg}
\end{array}
\right),
\\
\J_{58}&=\left(
\begin{array}{c|c}
-\RO_{eh} & 0 \\ \hline
0 & -\RO_{eh}
\end{array}
\right),\thinspace &
\J_{67}&=\left(
\begin{array}{c|c}
-\RO_{fg} & 0 \\ \hline
0 & -\RO_{fg}
\end{array}
\right),\thinspace &
\J_{68}&=\left(
\begin{array}{c|c}
-\RO_{fh} & 0 \\ \hline
0 & -\RO_{eg}
\end{array}
\right),\thinspace &
\J_{78}&=\left(
\begin{array}{c|c}
-\RO_{gh} & 0 \\ \hline
0 & -\RO_{gh}
\end{array}
\right),
\end{aligned}
\end{equation}
\end{footnotesize}
and the remaining $8$ complex structures are
\begin{footnotesize}
\begin{equation}\label{eq:J2}
\begin{aligned}
\J_{19}&=\left(
\begin{array}{c|c}
0 & -\Id \\ \hline
\Id & 0
\end{array}
\right),\thinspace &
\J_{29}&=\left(
\begin{array}{c|c}
0 & \RO_i \\ \hline
\RO_i & 0
\end{array}
\right),\thinspace &
\J_{39}&=\left(
\begin{array}{c|c}
0 & \RO_j \\ \hline
\RO_j & 0
\end{array}
\right),\thinspace &
\J_{49}&=\left(
\begin{array}{c|c}
0 & \RO_k \\ \hline
\RO_k & 0
\end{array}
\right),\\
\J_{59}&=\left(
\begin{array}{c|c}
0 & \RO_e \\ \hline
\RO_e & 0
\end{array}
\right),\thinspace &
\J_{69}&=\left(
\begin{array}{c|c}
0 & \RO_f \\ \hline
\RO_f & 0
\end{array}
\right),\thinspace &
\J_{79}&=\left(
\begin{array}{c|c}
0 & \RO_g \\ \hline
\RO_g & 0
\end{array}
\right),\thinspace &
\J_{89}&=\left(
\begin{array}{c|c}
0 & \RO_h \\ \hline
\RO_h & 0
\end{array}
\right).
\end{aligned}
\end{equation}
\end{footnotesize}

We write now the associated K\"ahler $2$-forms $\psi_{\alpha \beta}$ of the complex structures $\J_{\alpha \beta}$, by denoting the coordinates in $\OO^2\cong\RR^{16}$ by $(1,\dots,8,1',\dots,8')$. Abusing a little the notation introduced in Section~\ref{preliminaries}, we can then write
\begin{equation}\label{28}
\begin{aligned}
\psi_{12}&=(-\shortform{12}+\shortform{34}+\shortform{56}-\shortform{78})-(\enspace)\shortform{'}\enspace,\qquad & \psi_{13}&=(-\shortform{13}-\shortform{24}+\shortform{57}+\shortform{68})-(\enspace)\shortform{'}\enspace,
\\
\psi_{14}&=(-\shortform{14}+\shortform{23}+\shortform{58}-\shortform{67})-(\enspace)\shortform{'}\enspace,\qquad & \psi_{15}&=(-\shortform{15}-\shortform{26}-\shortform{37}-\shortform{48})-(\enspace)\shortform{'}\enspace,
\\
\psi_{16}&=(-\shortform{16}+\shortform{25}-\shortform{38}+\shortform{47})-(\enspace)\shortform{'}\enspace,\qquad & \psi_{17}&=(-\shortform{17}+\shortform{28}+\shortform{35}-\shortform{46})-(\enspace)\shortform{'}\enspace,
\\
\psi_{18}&=(-\shortform{18}-\shortform{27}+\shortform{36}+\shortform{45})-(\enspace)\shortform{'}\enspace,\qquad & \psi_{23}&=(-\shortform{14}+\shortform{23}-\shortform{58}+\shortform{67})+(\enspace)\shortform{'}\enspace,
\\
\psi_{24}&=(\shortform{13}+\shortform{24}+\shortform{57}+\shortform{68})+(\enspace)\shortform{'}\enspace,\qquad & \psi_{25}&=(-\shortform{16}+\shortform{25}+\shortform{38}-\shortform{47})+(\enspace)\shortform{'}\enspace,
\\
\psi_{26}&=(\shortform{15}+\shortform{26}-\shortform{37}-\shortform{48})+(\enspace)\shortform{'}\enspace,\qquad &
\psi_{27}&=(\shortform{18}+\shortform{27}+\shortform{36}+\shortform{45})+(\enspace)\shortform{'}\enspace,
\\
\psi_{28}&=(-\shortform{17}+\shortform{28}-\shortform{35}+\shortform{46})+(\enspace)\shortform{'}\enspace,\qquad & \psi_{34}&=(-\shortform{12}+\shortform{34}-\shortform{56}+\shortform{78})+(\enspace)\shortform{'}\enspace,
\\
\psi_{35}&=(-\shortform{17}-\shortform{28}+\shortform{35}+\shortform{46})+(\enspace)\shortform{'}\enspace,\qquad & \psi_{36}&=(-\shortform{18}+\shortform{27}+\shortform{36}-\shortform{45})+(\enspace)\shortform{'}\enspace,
\\
\psi_{37}&=(+\shortform{15}-\shortform{26}+\shortform{37}-\shortform{48})+(\enspace)\shortform{'}\enspace,\qquad & \psi_{38}&=(\shortform{16}+\shortform{25}+\shortform{38}+\shortform{47})+(\enspace)\shortform{'}\enspace,
\\
\psi_{45}&=(-\shortform{18}+\shortform{27}-\shortform{36}+\shortform{45})+(\enspace)\shortform{'}\enspace,\qquad & \psi_{46}&=(\shortform{17}+\shortform{28}+\shortform{35}+\shortform{46})+(\enspace)\shortform{'}\enspace,
\\
\psi_{47}&=(-\shortform{16}-\shortform{25}+\shortform{38}+\shortform{47})+(\enspace)\shortform{'}\enspace,\qquad & \psi_{48}&=(\shortform{15}-\shortform{26}-\shortform{37}+\shortform{48})+(\enspace)\shortform{'}\enspace,
\\
\psi_{56}&=(-\shortform{12}-\shortform{34}+\shortform{56}+\shortform{78})+(\enspace)\shortform{'}\enspace,\qquad & \psi_{57}&=(-\shortform{13}+\shortform{24}+\shortform{57}-\shortform{68})+(\enspace)\shortform{'}\enspace,
\\
\psi_{58}&=(-\shortform{14}-\shortform{23}+\shortform{58}+\shortform{67})+(\enspace)\shortform{'}\enspace,\qquad & \psi_{67}&=(\shortform{14}+\shortform{23}+\shortform{58}+\shortform{67})+(\enspace)\shortform{'}\enspace,
\\
\psi_{68}&=(-\shortform{13}+\shortform{24}-\shortform{57}+\shortform{68})+(\enspace)\shortform{'}\enspace,\qquad &
\psi_{78}&=(\shortform{12}+\shortform{34}+\shortform{56}+\shortform{78})+(\enspace)\shortform{'}\enspace,
\end{aligned}
\end{equation}
where $(\enspace)\shortform{'}$ denotes the $\shortform{'}$ of what appears before it, for instance
\[
\psi_{12}=(-\shortform{12}+\shortform{34}+\shortform{56}-\shortform{78})-(-\shortform{1'2'}+\shortform{3'4'}+\shortform{5'6'}-\shortform{7'8'})\enspace.
\]

Moreover we have
\begin{equation}\label{8}
\begin{aligned}
\psi_{19}&=-\shortform{11'}-\shortform{22'}-\shortform{33'}-\shortform{44'}-\shortform{55'}-\shortform{66'}-\shortform{77'}-\shortform{88'}\enspace, &
\psi_{29}&=-\shortform{12'}+\shortform{21'}+\shortform{34'}-\shortform{43'}+\shortform{56'}-\shortform{65'}-\shortform{78'}+\shortform{87'}\enspace,
\\
\psi_{39}&=-\shortform{13'}-\shortform{24'}+\shortform{31'}+\shortform{42'}+\shortform{57'}+\shortform{68'}-\shortform{75'}-\shortform{86'}\enspace, &
\psi_{49}&=-\shortform{14'}+\shortform{23'}-\shortform{32'}+\shortform{41'}+\shortform{58'}-\shortform{67'}+\shortform{76'}-\shortform{85'}\enspace,
\\
\psi_{59}&=-\shortform{15'}-\shortform{26'}-\shortform{37'}-\shortform{48'}+\shortform{51'}+\shortform{62'}+\shortform{73'}+\shortform{84'}\enspace, &
\psi_{69}&=-\shortform{16'}+\shortform{25'}-\shortform{38'}+\shortform{47'}-\shortform{52'}+\shortform{61'}-\shortform{74'}+\shortform{83'}\enspace,
\\
\psi_{79}&=-\shortform{17'}+\shortform{28'}+\shortform{35'}-\shortform{46'}-\shortform{53'}+\shortform{64'}+\shortform{71'}-\shortform{82'}\enspace, &
\psi_{89}&=-\shortform{18'}-\shortform{27'}+\shortform{36'}+\shortform{45'}-\shortform{54'}-\shortform{63'}+\shortform{72'}+\shortform{81'}\enspace.\end{aligned}
\end{equation}

\begin{pr}\label{pr:charpoly}
Let $\psi$ be the matrix of K\"ahler forms $(\psi_{\alpha \beta})_{\alpha,\beta=1,\dots,9}$ given by Formulas~\eqref{28} and~\eqref{8}. Then its characteristic polynomial reduces to
\[
\det(tI-\psi)=t^9+\tau_4(\psi)t^5+\tau_8(\psi)t\enspace.
\]
\end{pr}

\begin{proof}
The coefficients $\tau_{2k-1}$ of $\det(tI-\psi)$ are zero, since $\psi$ is a $9\times 9$ skew-symmetric matrix. Thus, it remains to check that $\tau_2=\tau_6=0$. Since $\tau_6(\psi)$ is the Hodge star of $\tau_2(\psi)$ in $\RR^{16}$, we are only left to show the vanishing of $\tau_2(\psi)$.

Observe that the K\"ahler forms in~\eqref{28} can be arranged in the following seven families:
\begin{equation}
\begin{split}
\psi_{12},\psi_{34},\psi_{56}, \psi_{78}&=[\pm(\shortform{12})\pm(\shortform{34})\pm(\shortform{56})\pm(\shortform{78})]\pm[\pm(\shortform{12})\shortform{'}\pm(\shortform{34})\shortform{'}\pm(\shortform{56})\shortform{'}\pm(\shortform{78})\shortform{'}]\enspace,
\\
\psi_{13},\psi_{24},\psi_{57},\psi_{68}&=[\pm(\shortform{13})\pm(\shortform{24})\pm(\shortform{57})\pm(\shortform{68})]\pm[\pm(\shortform{13})\shortform{'}\pm(\shortform{24})\shortform{'}\pm(\shortform{57})\shortform{'}\pm(\shortform{68})\shortform{'}]\enspace,
\\
\psi_{14},\psi_{23},\psi_{58},\psi_{67}&=[\pm(\shortform{14})\pm(\shortform{23})\pm(\shortform{58})\pm(\shortform{67})]\pm[\pm(\shortform{14})\shortform{'}\pm(\shortform{23})\shortform{'}\pm(\shortform{58})\shortform{'}\pm(\shortform{67})\shortform{'}]\enspace,
\\
\psi_{15},\psi_{26},\psi_{37},\psi_{48}&=[\pm(\shortform{15})\pm(\shortform{26})\pm(\shortform{37})\pm(\shortform{48})]\pm[\pm(\shortform{15})\shortform{'}\pm(\shortform{26})\shortform{'}\pm(\shortform{37})\shortform{'}\pm(\shortform{48})\shortform{'}]\enspace,
\\
\psi_{16},\psi_{25},\psi_{38},\psi_{47}&=[\pm(\shortform{16})\pm(\shortform{25})\pm(\shortform{38})\pm(\shortform{47})]\pm[\pm(\shortform{16})\shortform{'}\pm(\shortform{25})\shortform{'}\pm(\shortform{38})\shortform{'}\pm(\shortform{47})\shortform{'}]\enspace,
\\
\psi_{17},\psi_{28},\psi_{35},\psi_{46}&=[\pm(\shortform{17})\pm(\shortform{28})\pm(\shortform{35})\pm(\shortform{46})]\pm[\pm(\shortform{17})\shortform{'}\pm(\shortform{28})\shortform{'}\pm(\shortform{35})\shortform{'}\pm(\shortform{46})\shortform{'}]\enspace,
\\
\psi_{18},\psi_{27},\psi_{36},\psi_{45}&=[\pm(\shortform{18})\pm(\shortform{27})\pm(\shortform{36})\pm(\shortform{45})]\pm[\pm(\shortform{18})\shortform{'}\pm(\shortform{27})\shortform{'}\pm(\shortform{36})\shortform{'}\pm(\shortform{45})\shortform{'}]\enspace,
\end{split}
\end{equation}
and note that the four $\psi_{\alpha \beta}$ in each line indicate precisely the four pairs $(\shortform{\alpha\beta})$ appearing in their expression. Note also that in each line the signs inside brackets follow (up to a global change) the four patterns
\begin{equation}
\begin{split}
-&(\phantom{\shortform{12}})+(\phantom{\shortform{12}})+(\phantom{\shortform{12}})-(\phantom{\shortform{12}})\enspace,\qquad -(\phantom{\shortform{12}})+(\phantom{\shortform{12}})-(\phantom{\shortform{12}})+(\phantom{\shortform{12}})\enspace, 
\\
-&(\phantom{\shortform{12}})-(\phantom{\shortform{12}})+(\phantom{\shortform{12}})+(\phantom{\shortform{12}})\enspace,\qquad +(\phantom{\shortform{12}})+(\phantom{\shortform{12}})+(\phantom{\shortform{12}})+(\phantom{\shortform{12}})\enspace,
\end{split}
\end{equation}
that is, an even number of $+$ and $-$ signs. Finally, observe that in all cases (and again up to a global change for forms of type $\psi_{1 \beta}$, for $\beta=2,\dots,8$) the same pattern appears both in the terms with coordinates $(1,\dots,8)$ and in those with coordinates $(1',\dots,8')$. 

It follows
\begin{equation}\label{7sums}
\begin{split}
\frac{1}{4}(\psi^2_{12}+\psi^2_{34}+\psi^2_{56}+\psi^2_{78})&=\underline{\shortform{121'2'}}+\shortform{123'4'}+\shortform{125'6'}-\shortform{127'8'}+\shortform{341'2'}+\underline{\shortform{343'4'}}-\shortform{345'6'}+\shortform{347'8'}
\\
&+\shortform{561'2'}-\shortform{563'4'}+\underline{\shortform{565'6'}}+\shortform{567'8'}-\shortform{781'2'}+\shortform{783'4'}+\shortform{785'6'}+\underline{\shortform{787'8'}}\enspace,
\\
\frac{1}{4}(\psi^2_{13}+\psi^2_{24}+\psi^2_{57}+\psi^2_{68})&=\underline{\shortform{131'3'}}-\shortform{132'4'}+\shortform{135'7'}+\shortform{136'8'}-\shortform{241'3'}+\underline{\shortform{242'4'}}+\shortform{245'7'}+\shortform{246'8'}
\\
&+\shortform{571'3'}+\shortform{572'4'}+\underline{\shortform{575'7'}}-\shortform{576'8'}+\shortform{681'3'}+\shortform{682'4'}-\shortform{685'7'}+\underline{\shortform{686'8'}}\enspace,
\\
\frac{1}{4}(\psi^2_{14}+\psi^2_{23}+\psi^2_{58}+\psi^2_{67})&=\underline{\shortform{141'4'}}+\shortform{142'3'}+\shortform{145'8'}-\shortform{146'7'}+\shortform{231'4'}+\underline{\shortform{232'3'}}-\shortform{235'8'}+\shortform{236'7'}
\\
&+\shortform{581'4'}-\shortform{582'3'}+\underline{\shortform{585'8'}}+\shortform{586'7'}-\shortform{671'4'}+\shortform{672'3'}+\shortform{675'8'}+\underline{\shortform{676'7'}}\enspace,
\\
\frac{1}{4}(\psi^2_{15}+\psi^2_{26}+\psi^2_{37}+\psi^2_{48})&=\underline{\shortform{151'5'}}-\shortform{152'6'}-\shortform{153'7'}-\shortform{154'8'}-\shortform{261'5'}+\underline{\shortform{262'6'}}-\shortform{263'7'}-\shortform{264'8'}
\\
&-\shortform{371'5'}-\shortform{372'6'}+\underline{\shortform{373'7'}}-\shortform{374'8'}-\shortform{481'5'}-\shortform{482'6'}-\shortform{483'7'}+\underline{\shortform{484'8'}}\enspace,
\\
\frac{1}{4}(\psi^2_{16}+\psi^2_{25}+\psi^2_{38}+\psi^2_{47})&=\underline{\shortform{161'6'}}+\shortform{162'5'}-\shortform{163'8'}+\shortform{164'7'}+\shortform{251'6'}+\underline{\shortform{252'5'}}+\shortform{253'8'}-\shortform{254'7'}
\\
&-\shortform{381'6'}+\shortform{382'5'}+\underline{\shortform{383'8'}}+\shortform{384'7'}+\shortform{471'6'}-\shortform{472'5'}+473'8+\underline{\shortform{474'7'}}\enspace,
\\
\frac{1}{4}(\psi^2_{17}+\psi^2_{28}+\psi^2_{35}+\psi^2_{46})&=\underline{\shortform{171'7'}}+\shortform{172'8'}+\shortform{173'5'}-\shortform{174'6'}+\shortform{281'7'}+\underline{\shortform{282'8'}}-\shortform{283'5'}+\shortform{284'6'}
\\
&+\shortform{351'7'}-\shortform{352'8'}+\underline{\shortform{353'5'}}+\shortform{354'6'}-\shortform{461'7'}+\shortform{462'8'}+\shortform{463'5'}+\underline{\shortform{464'6'}}\enspace,
\\
\frac{1}{4}(\psi^2_{18}+\psi^2_{27}+\psi^2_{36}+\psi^2_{45})&=\underline{\shortform{181'8'}}-\shortform{182'7'}+\shortform{183'6'}+\shortform{184'5'}-\shortform{271'8'}+\underline{\shortform{272'7'}}+\shortform{273'6'}+\shortform{274'5'}
\\
&+\shortform{361'8'}+\shortform{362'7'}+\underline{\shortform{363'6'}}-\shortform{364'5'}+\shortform{451'8'}+\shortform{452'7'}-\shortform{453'6'}+\underline{\shortform{454'5'}}\enspace,
\end{split}
\end{equation}
where the sum of the underlined terms is equal to $-\frac{1}{2}\psi^2_{19}$. On the other hand the computation of
\begin{equation}
\frac{1}{2}(\psi^2_{29}+\psi^2_{39}+\psi^2_{49}+\psi^2_{59}+\psi^2_{69}+\psi^2_{79}+\psi^2_{89}) 
\end{equation}
through formulas~\eqref{8} yields a sum of $28 \times 7 = 196$ terms. Among them, the $28$ involving repeated coordinates (like $-\shortform{12'21'}$ or $-\shortform{34'43'}$) reproduce the negative of the underlined terms in~\ref{7sums} (or equivalently the terms of $-\frac{1}{2}\psi^2_{19}$). Moreover, the remaining $196-28= 168$ terms are in $2$ to $1$ correspondence with the negative of the $84$ non-underlined terms in~\ref{7sums}. In other words, we have
\begin{equation*}
\frac{1}{4}(\sum_{\alpha<\beta=2}^{8}\psi^2_{\alpha\beta}+\psi^2_{19}+\sum_{\alpha=2}^{8}\psi^2_{\alpha 9})=0\enspace,
\end{equation*}
which gives (cf. ~\cite{cgm} as well as the discussion in~\cite{br-gr} concerning invariant tensors of the $16$-dimensional representation of $\Spin{9}$):
\begin{equation}
\tau_2(\psi)=\sum_{\alpha<\beta=2}^{9}\psi^2_{\alpha\beta}=0\enspace.
\end{equation}
\end{proof}

Since $\tau_8(\psi)$ is a $16$-form, it is proportional to the volume form of $\RR^{16}$. Sections~\ref{sec:forma_spin9} and ~\ref{sec:charpoly} will be devoted to the computation of $\tau_4(\psi)$.

\section{The $8$-form of a $\Spin{9}$-structure}\label{sec:forma_spin9}

Recall (cf.~\cite[page 13]{be} as well as~\cite[pages 168--170]{co}) that a $\Spin{9}$-invariant $8$-form $\spinform{9}$ in $\OO^2$ can be defined through the projections $p_l$ from $\OO^2$ onto its octonionic lines $l$. If $\nu_l$ is the volume form on each line $l$, then
\begin{equation}\label{be}
\spinform{9}\ug c\int_{\OP{1}}p_l^*\nu_l\,dl\enspace,
\end{equation}
for some constant $c$. The $\Spin{9}$-invariance of $\spinform{9}$ is a consequence of the $\Spin{9}$ action on the octonionic lines $l$ and of the $\Spin{9}$-invariance of the measure $dl$ on $\OP{1}$.

The above definition of $\spinform{9}$, forerunning the point of view of calibrations, parallels that of the K\"ahler $2$-form in $\CC^n$ and the quaternionic $4$-form in $\HH^n$ as the integrals $\int_{\CP{n-1}}p_l^*\nu_l\,dl$ and $\int_{\HP{n-1}}p_l^*\nu_l\,dl$ respectively. 

In this Section, we will use Formula~\eqref{be} to explicitly compute $\spinform{9}\in\Lambda^8\RR^{16}$, and we will give a formula analogous to Formula~\eqref{Phi}. To this aim, it is convenient to look first at the corresponding approaches for the complex and quaternionic cases.

\subsection*{The K\"ahler form in $\CC^2$}

Let $l\ug l_m=\{(z, mz)\st z\in\CC\}$ be a complex line in $\CC^2$, where $m=m_1+im_2\in\CC$. Denote by $p_l:\CC^2-0\rightarrow l_m$ the projection, and by $\nu_l$ the volume form on $l_m\subset\RR^4$. Writing the generators $(1,m)$ and $(i,mi)$ of $l_m$ in real coordinates, and using again the notation $\{\shortform{1},\dots,\shortform{4}\}$ for the standard basis of $\Lambda^1\RR^4$, we obtain the following coframe $\alpha_1$, $\alpha_2$ in $l_m\subset\RR^4$:
\[
\begin{split}
\alpha_1&=\shortform{1}+m_1 \shortform{3}+m_2 \shortform{4}\enspace,\\
\alpha_2&=\shortform{2}-m_2\shortform{3}+m_1\shortform{4}\enspace.
\end{split}
\]
Thus, we have
\[
\begin{split}
\kahlerform\ug c\int_{\CP{1}}p_{l}^*\nu_{l}\,dl&=c\int_{\RR^2}\frac{\alpha_1\wedge\alpha_2}{1+m_1^2+m_2^2}\,dm_1\wedge dm_2 \\
&=c\int_{\RR^2}\frac{(\shortform{1}+m_1\shortform{3}+m_2\shortform{4})\wedge(\shortform{2}-m_2\shortform{3}+m_1\shortform{4})}{1+m_1^2+m_2^2}\,dm_1\wedge dm_2\enspace,
\end{split}
\]
and using polar coordinates $m=\rho e^{i\theta}$, we have $dm_1\wedge dm_2=\rho/(1+\rho^2)^2 d\rho\wedge d\theta$ and
\[
\kahlerform=c\int_0^\infty\int_0^{2\pi}\frac{(\shortform{1}+\rho\cos\theta\shortform{3}+\rho\sin\theta\shortform{4})\wedge(\shortform{2}-\rho\sin\theta\shortform{3}+\rho\cos\theta\shortform{4})\rho}{(1+\rho^2)^3}\,d\rho\wedge d\theta\enspace.
\]
The $2$-form $\kahlerform$ is then described by $\binom{4}{2}$ integrals, and a computation shows that the only non-zero coefficients are that of $\shortform{12}$ and $\shortform{34}$:
\[
\int_0^\infty\int_0^{2\pi}\rho/(1+\rho^2)^3\,d\rho\wedge d\theta=\pi/2=\int_0^\infty\int_0^{2\pi}\rho^3/(1+\rho^2)^3\,d\rho\wedge d\theta.
\]

Thus, for a ad-hoc choice of the constant $c$, we obtain
\begin{equation}\label{integralformC}
\kahlerform=(2/\pi)\int_{\CP{1}}p_l^*\nu_l\,dl=\shortform{12}+\shortform{34}=\text{K\"ahler form in $\CC^2$}\enspace.
\end{equation}

\subsection*{The quaternion-K\"ahler form in $\HH^2$}

Following the complex case, in the quaternionic case we write $l\ug l_m =\{(h,mh)\st h\in\HH\}$, where $m=m_1+im_2+jm_3+km_4\in\HH$. Denote again by $p_l:\HH^2-0\rightarrow l_m$ the projection, by $\nu_l$ the volume form on $l_m\subset\RR^8$ and by $\{\shortform{1},\dots,\shortform{8}\}$ the standard real coframe of $\RR^8$. The coframe $\{\alpha_1,\dots,\alpha_4\}$, dual to $\{(1,m),(i,mi),(j,mj),(k,mk)\}$ in $l_m\subset\RR^8$, is then given by
\[
\begin{split}
\alpha_1&=\shortform{1}+m_1 \shortform{5}+m_2 \shortform{6}+m_3 \shortform{7}+m_4 \shortform{8}\enspace, \\ 
\alpha_2&=\shortform{2}-m_2 \shortform{5}+m_1 \shortform{6}+m_4 \shortform{7}-m_3 \shortform{8}\enspace, \\
\alpha_3&=\shortform{3}-m_3 \shortform{5}-m_4 \shortform{6}+m_1 \shortform{7}+m_2 \shortform{8}\enspace, \\
\alpha_4&=\shortform{4}-m_4 \shortform{5}+m_3 \shortform{6}-m_2 \shortform{7}+m_1 \shortform{8}\enspace,
\end{split}
\]
and the integral $4$-form is
\begin{equation}\label{quaternionkahlerform}
\quaternionform\ug c\int_{\HP{1}}p_{l}^*\nu_{l}\,dl=\int_{\RR^4}\frac{\alpha_1\wedge\dots\wedge\alpha_4}{(1+m_1^2+m_2^2+m_3^2+m_4^2)^2}\,dm_1\wedge\dots\wedge dm_4\enspace.
\end{equation}

Again, the computation of these $\binom{8}{4}$ integrals can be done in polar coordinates. We used \Mathematica\ for this computation, obtaining
\[
\begin{split}
\quaternionform=c\left(\frac{\pi^2}{20}\right.&\shortform{1234}-\frac{\pi^2}{60}\shortform{1256}+\frac{\pi^2}{60}\shortform{1278}-\frac{\pi^2}{60}\shortform{1357}-\frac{\pi^2}{60}\shortform{1368}-\frac{\pi^2}{60}\shortform{1458}+\frac{\pi^2}{60}\shortform{1467}\\
+\frac{\pi^2}{60}&\shortform{2358}-\frac{\pi^2}{60}\shortform{2367}-\frac{\pi^2}{60}\shortform{2457}-\frac{\pi^2}{60}\shortform{2468}+\frac{\pi^2}{60}\shortform{3456}-\frac{\pi^2}{60}\shortform{3478}+\left.\frac{\pi^2}{20}\shortform{5678}\right)\enspace,
\end{split}
\]
and a comparison with the right quaternionic $4$-form $\omega_{\RH_i}^2+\omega_{\RH_j}^2+\omega_{\RH_k}^2$ (see also Proposition~\ref{pr:quaternion}) leads to
\begin{equation}\label{integralformH}
\quaternionform=(-120/\pi^2)\int_{\HP{1}}p_l^*\nu_l\,dl=\omega_{\RH_i}^2+\omega_{\RH_j}^2+\omega_{\RH_k}^2=\text{Right quaternion-K\"ahler form in $\HH^2$}\enspace.
\end{equation}

\subsection*{The $\Spin{9}$-form in $\OO^2$}

In the octonionic case we write $l\ug l_m =\{(x,mx)\st x\in\OO\}$, where $m=m_1+im_2+jm_3+km_4+em_5+fm_6+gm_7+hm_8\in\OO$. The projection is $p_l:\OO^2-0\rightarrow l_m$, and $\nu_l$ is the volume form on $l_m\subset\RR^{16}$. For the sake of notation, it is now convenient to split $\RR^{16}$ as two copies of $\RR^8$, so to denote by $\{\shortform{1},\dots,\shortform{8},\shortform{1'},\dots,\shortform{8'}\}$ the standard real coframe of $\RR^{16}$. We have already introduced this notation in Formula~\eqref{28}.

In the same way as before, we obtain the coframe $\{\alpha_1,\dots,\alpha_8\}$ dual to $\{(1,m),\dots,(h,mh)\}$ in $l_m\subset\RR^{16}$, and the integral $8$-form~\eqref{be} is
\begin{equation}\label{spin9form}
\spinform{9}=c\int_{\OP{1}}p_{l}^*\nu_{l}\,dl=\int_{\RR^8}\frac{\alpha_1\wedge\dots\wedge\alpha_8}{(1+m_1^2+\dots+m_8^2)^2}\,dm_1\wedge\dots\wedge dm_8\enspace.
\end{equation}

A \Mathematica\ computation in polar coordinates of these $\binom{16}{8}$ integrals leads then to the explicit formula, given in Table~\ref{explicitspin9}. Remark that since $\spinform{9}$ is self-dual, Table~\ref{explicitspin9} lists only half of its monomials, the remaining ones being their Hodge stars. Thus the content of Table~\ref{explicitspin9} should be looked at as the analogue to formula~\eqref{Phi} for the $4$-form $\spinform{7}$.

Recall also that the entries of Table~\ref{explicitspin9} have been computed, according to the theorem in the Introduction, in such a way that the coefficients of $\spinform{9}$ be integers with $\gcd=1$. Thus, with this constraint on the constant $c$, we have
\begin{equation}\label{eq:implicitspin9}
\spinform{9}=\frac{110880}{\pi^4}\int_{\OP{1}}p_{l}^*\nu_{l}\,dl\enspace.
\end{equation}



\begin{table}[H]
\centering
\resizebox*{0.9\textwidth}{!}{
\begin{tabular}{||rrr||rrr||rrr||rrr||rrr||}
\hline
\shortform{12345678} & & -14 &  \shortform{123456} & \shortform{1'2'} & 2 &  \shortform{123456} & \shortform{3'4'} & -2 &  \shortform{123456} & \shortform{5'6'} & -2 &  \shortform{123456} & \shortform{7'8'} & -2 \\ 
\shortform{123457} & \shortform{1'3'} & 2 &  \shortform{123457} & \shortform{2'4'} & 2 &  \shortform{123457} & \shortform{5'7'} & -2 &  \shortform{123457} & \shortform{6'8'} & 2 &  \shortform{123458} & \shortform{1'4'} & 2 \\ 
\shortform{123458} & \shortform{2'3'} & -2 &  \shortform{123458} & \shortform{5'8'} & -2 &  \shortform{123458} & \shortform{6'7'} & -2 &  \shortform{123467} & \shortform{1'4'} & -2 &  \shortform{123467} & \shortform{2'3'} & 2 \\ 
\shortform{123467} & \shortform{5'8'} & -2 &  \shortform{123467} & \shortform{6'7'} & -2 &  \shortform{123468} & \shortform{1'3'} & 2 &  \shortform{123468} & \shortform{2'4'} & 2 &  \shortform{123468} & \shortform{5'7'} & 2 \\ 
\shortform{123468} & \shortform{6'8'} & -2 &  \shortform{123478} & \shortform{1'2'} & -2 &  \shortform{123478} & \shortform{3'4'} & 2 &  \shortform{123478} & \shortform{5'6'} & -2 &  \shortform{123478} & \shortform{7'8'} & -2 \\ 
\shortform{1234} & \shortform{1'2'3'4'} & -2 &  \shortform{1234} & \shortform{5'6'7'8'} & -2 &  \shortform{123567} & \shortform{1'5'} & -2 &  \shortform{123567} & \shortform{2'6'} & -2 &  \shortform{123567} & \shortform{3'7'} & -2 \\ 
\shortform{123567} & \shortform{4'8'} & 2 &  \shortform{123568} & \shortform{1'6'} & -2 &  \shortform{123568} & \shortform{2'5'} & 2 &  \shortform{123568} & \shortform{3'8'} & -2 &  \shortform{123568} & \shortform{4'7'} & -2 \\ 
\shortform{123578} & \shortform{1'7'} & -2 &  \shortform{123578} & \shortform{2'8'} & 2 &  \shortform{123578} & \shortform{3'5'} & 2 &  \shortform{123578} & \shortform{4'6'} & 2 &  \shortform{1235} & \shortform{1'2'3'5'} & -1 \\ 
\shortform{1235} & \shortform{1'2'4'6'} & -1 &  \shortform{1235} & \shortform{1'3'4'7'} & -1 &  \shortform{1235} & \shortform{1'5'6'7'} & -1 &  \shortform{1235} & \shortform{2'3'4'8'} & 1 &  \shortform{1235} & \shortform{2'5'6'8'} & 1 \\ 
\shortform{1235} & \shortform{3'5'7'8'} & 1 &  \shortform{1235} & \shortform{4'6'7'8'} & 1 &  \shortform{123678} & \shortform{1'8'} & -2 &  \shortform{123678} & \shortform{2'7'} & -2 &  \shortform{123678} & \shortform{3'6'} & 2 \\ 
\shortform{123678} & \shortform{4'5'} & -2 &  \shortform{1236} & \shortform{1'2'3'6'} & -1 &  \shortform{1236} & \shortform{1'2'4'5'} & 1 &  \shortform{1236} & \shortform{1'3'4'8'} & -1 &  \shortform{1236} & \shortform{1'5'6'8'} & -1 \\ 
\shortform{1236} & \shortform{2'3'4'7'} & -1 &  \shortform{1236} & \shortform{2'5'6'7'} & -1 &  \shortform{1236} & \shortform{3'6'7'8'} & 1 &  \shortform{1236} & \shortform{4'5'7'8'} & -1 &  \shortform{1237} & \shortform{1'2'3'7'} & -1 \\ 
\shortform{1237} & \shortform{1'2'4'8'} & 1 &  \shortform{1237} & \shortform{1'3'4'5'} & 1 &  \shortform{1237} & \shortform{1'5'7'8'} & -1 &  \shortform{1237} & \shortform{2'3'4'6'} & 1 &  \shortform{1237} & \shortform{2'6'7'8'} & -1 \\ 
\shortform{1237} & \shortform{3'5'6'7'} & -1 &  \shortform{1237} & \shortform{4'5'6'8'} & 1 &  \shortform{1238} & \shortform{1'2'3'8'} & -1 &  \shortform{1238} & \shortform{1'2'4'7'} & -1 &  \shortform{1238} & \shortform{1'3'4'6'} & 1 \\ 
\shortform{1238} & \shortform{1'6'7'8'} & -1 &  \shortform{1238} & \shortform{2'3'4'5'} & -1 &  \shortform{1238} & \shortform{2'5'7'8'} & 1 &  \shortform{1238} & \shortform{3'5'6'8'} & -1 &  \shortform{1238} & \shortform{4'5'6'7'} & -1 \\ 
\shortform{124567} & \shortform{1'6'} & 2 &  \shortform{124567} & \shortform{2'5'} & -2 &  \shortform{124567} & \shortform{3'8'} & -2 &  \shortform{124567} & \shortform{4'7'} & -2 &  \shortform{124568} & \shortform{1'5'} & -2 \\ 
\shortform{124568} & \shortform{2'6'} & -2 &  \shortform{124568} & \shortform{3'7'} & 2 &  \shortform{124568} & \shortform{4'8'} & -2 &  \shortform{124578} & \shortform{1'8'} & -2 &  \shortform{124578} & \shortform{2'7'} & -2 \\ 
\shortform{124578} & \shortform{3'6'} & -2 &  \shortform{124578} & \shortform{4'5'} & 2 &  \shortform{1245} & \shortform{1'2'3'6'} & 1 &  \shortform{1245} & \shortform{1'2'4'5'} & -1 &  \shortform{1245} & \shortform{1'3'4'8'} & -1 \\ 
\shortform{1245} & \shortform{1'5'6'8'} & -1 &  \shortform{1245} & \shortform{2'3'4'7'} & -1 &  \shortform{1245} & \shortform{2'5'6'7'} & -1 &  \shortform{1245} & \shortform{3'6'7'8'} & -1 &  \shortform{1245} & \shortform{4'5'7'8'} & 1 \\ 
\shortform{124678} & \shortform{1'7'} & 2 &  \shortform{124678} & \shortform{2'8'} & -2 &  \shortform{124678} & \shortform{3'5'} & 2 &  \shortform{124678} & \shortform{4'6'} & 2 &  \shortform{1246} & \shortform{1'2'3'5'} & -1 \\ 
\shortform{1246} & \shortform{1'2'4'6'} & -1 &  \shortform{1246} & \shortform{1'3'4'7'} & 1 &  \shortform{1246} & \shortform{1'5'6'7'} & 1 &  \shortform{1246} & \shortform{2'3'4'8'} & -1 &  \shortform{1246} & \shortform{2'5'6'8'} & -1 \\ 
\shortform{1246} & \shortform{3'5'7'8'} & 1 &  \shortform{1246} & \shortform{4'6'7'8'} & 1 &  \shortform{1247} & \shortform{1'2'3'8'} & -1 &  \shortform{1247} & \shortform{1'2'4'7'} & -1 &  \shortform{1247} & \shortform{1'3'4'6'} & -1 \\ 
\shortform{1247} & \shortform{1'6'7'8'} & 1 &  \shortform{1247} & \shortform{2'3'4'5'} & 1 &  \shortform{1247} & \shortform{2'5'7'8'} & -1 &  \shortform{1247} & \shortform{3'5'6'8'} & -1 &  \shortform{1247} & \shortform{4'5'6'7'} & -1 \\ 
\shortform{1248} & \shortform{1'2'3'7'} & 1 &  \shortform{1248} & \shortform{1'2'4'8'} & -1 &  \shortform{1248} & \shortform{1'3'4'5'} & 1 &  \shortform{1248} & \shortform{1'5'7'8'} & -1 &  \shortform{1248} & \shortform{2'3'4'6'} & 1 \\ 
\shortform{1248} & \shortform{2'6'7'8'} & -1 &  \shortform{1248} & \shortform{3'5'6'7'} & 1 &  \shortform{1248} & \shortform{4'5'6'8'} & -1 &  \shortform{125678} & \shortform{1'2'} & -2 &  \shortform{125678} & \shortform{3'4'} & -2 \\ 
\shortform{125678} & \shortform{5'6'} & 2 &  \shortform{125678} & \shortform{7'8'} & -2 &  \shortform{1256} & \shortform{1'2'5'6'} & -2 &  \shortform{1256} & \shortform{3'4'7'8'} & -2 &  \shortform{1257} & \shortform{1'2'5'7'} & -1 \\ 
\shortform{1257} & \shortform{1'2'6'8'} & 1 &  \shortform{1257} & \shortform{1'3'5'6'} & -1 &  \shortform{1257} & \shortform{1'3'7'8'} & 1 &  \shortform{1257} & \shortform{2'4'5'6'} & -1 &  \shortform{1257} & \shortform{2'4'7'8'} & 1 \\ 
\shortform{1257} & \shortform{3'4'5'7'} & -1 &  \shortform{1257} & \shortform{3'4'6'8'} & 1 &  \shortform{1258} & \shortform{1'2'5'8'} & -1 &  \shortform{1258} & \shortform{1'2'6'7'} & -1 &  \shortform{1258} & \shortform{1'4'5'6'} & -1 \\ 
\shortform{1258} & \shortform{1'4'7'8'} & 1 &  \shortform{1258} & \shortform{2'3'5'6'} & 1 &  \shortform{1258} & \shortform{2'3'7'8'} & -1 &  \shortform{1258} & \shortform{3'4'5'8'} & -1 &  \shortform{1258} & \shortform{3'4'6'7'} & -1 \\ 
\shortform{1267} & \shortform{1'2'5'8'} & -1 &  \shortform{1267} & \shortform{1'2'6'7'} & -1 &  \shortform{1267} & \shortform{1'4'5'6'} & 1 &  \shortform{1267} & \shortform{1'4'7'8'} & -1 &  \shortform{1267} & \shortform{2'3'5'6'} & -1 \\ 
\shortform{1267} & \shortform{2'3'7'8'} & 1 &  \shortform{1267} & \shortform{3'4'5'8'} & -1 &  \shortform{1267} & \shortform{3'4'6'7'} & -1 &  \shortform{1268} & \shortform{1'2'5'7'} & 1 &  \shortform{1268} & \shortform{1'2'6'8'} & -1 \\ 
\shortform{1268} & \shortform{1'3'5'6'} & -1 &  \shortform{1268} & \shortform{1'3'7'8'} & 1 &  \shortform{1268} & \shortform{2'4'5'6'} & -1 &  \shortform{1268} & \shortform{2'4'7'8'} & 1 &  \shortform{1268} & \shortform{3'4'5'7'} & 1 \\ 
\shortform{1268} & \shortform{3'4'6'8'} & -1 &  \shortform{1278} & \shortform{1'2'7'8'} & -2 &  \shortform{1278} & \shortform{3'4'5'6'} & -2 &  \shortform{12} & \shortform{1'2'3'4'5'6'} & 2 &  \shortform{12} & \shortform{1'2'3'4'7'8'} & -2 \\ 
\shortform{12} & \shortform{1'2'5'6'7'8'} & -2 &  \shortform{12} & \shortform{3'4'5'6'7'8'} & -2 &  \shortform{134567} & \shortform{1'7'} & 2 &  \shortform{134567} & \shortform{2'8'} & 2 &  \shortform{134567} & \shortform{3'5'} & -2 \\ 
\shortform{134567} & \shortform{4'6'} & 2 &  \shortform{134568} & \shortform{1'8'} & 2 &  \shortform{134568} & \shortform{2'7'} & -2 &  \shortform{134568} & \shortform{3'6'} & -2 &  \shortform{134568} & \shortform{4'5'} & -2 \\ 
\shortform{134578} & \shortform{1'5'} & -2 &  \shortform{134578} & \shortform{2'6'} & 2 &  \shortform{134578} & \shortform{3'7'} & -2 &  \shortform{134578} & \shortform{4'8'} & -2 &  \shortform{1345} & \shortform{1'2'3'7'} & 1 \\ 
\shortform{1345} & \shortform{1'2'4'8'} & 1 &  \shortform{1345} & \shortform{1'3'4'5'} & -1 &  \shortform{1345} & \shortform{1'5'7'8'} & -1 &  \shortform{1345} & \shortform{2'3'4'6'} & 1 &  \shortform{1345} & \shortform{2'6'7'8'} & 1 \\ 
\shortform{1345} & \shortform{3'5'6'7'} & -1 &  \shortform{1345} & \shortform{4'5'6'8'} & -1 &  \shortform{134678} & \shortform{1'6'} & -2 &  \shortform{134678} & \shortform{2'5'} & -2 &  \shortform{134678} & \shortform{3'8'} & -2 \\ 
\shortform{134678} & \shortform{4'7'} & 2 &  \shortform{1346} & \shortform{1'2'3'8'} & 1 &  \shortform{1346} & \shortform{1'2'4'7'} & -1 &  \shortform{1346} & \shortform{1'3'4'6'} & -1 &  \shortform{1346} & \shortform{1'6'7'8'} & -1 \\ 
\shortform{1346} & \shortform{2'3'4'5'} & -1 &  \shortform{1346} & \shortform{2'5'7'8'} & -1 &  \shortform{1346} & \shortform{3'5'6'8'} & -1 &  \shortform{1346} & \shortform{4'5'6'7'} & 1 &  \shortform{1347} & \shortform{1'2'3'5'} & -1 \\ 
\shortform{1347} & \shortform{1'2'4'6'} & 1 &  \shortform{1347} & \shortform{1'3'4'7'} & -1 &  \shortform{1347} & \shortform{1'5'6'7'} & 1 &  \shortform{1347} & \shortform{2'3'4'8'} & -1 &  \shortform{1347} & \shortform{2'5'6'8'} & 1 \\ 
\shortform{1347} & \shortform{3'5'7'8'} & -1 &  \shortform{1347} & \shortform{4'6'7'8'} & 1 &  \shortform{1348} & \shortform{1'2'3'6'} & -1 &  \shortform{1348} & \shortform{1'2'4'5'} & -1 &  \shortform{1348} & \shortform{1'3'4'8'} & -1 \\ 
\shortform{1348} & \shortform{1'5'6'8'} & 1 &  \shortform{1348} & \shortform{2'3'4'7'} & 1 &  \shortform{1348} & \shortform{2'5'6'7'} & -1 &  \shortform{1348} & \shortform{3'6'7'8'} & -1 &  \shortform{1348} & \shortform{4'5'7'8'} & -1 \\ 
\shortform{135678} & \shortform{1'3'} & -2 &  \shortform{135678} & \shortform{2'4'} & 2 &  \shortform{135678} & \shortform{5'7'} & 2 &  \shortform{135678} & \shortform{6'8'} & 2 &  \shortform{1356} & \shortform{1'2'5'7'} & -1 \\ 
\shortform{1356} & \shortform{1'2'6'8'} & -1 &  \shortform{1356} & \shortform{1'3'5'6'} & -1 &  \shortform{1356} & \shortform{1'3'7'8'} & -1 &  \shortform{1356} & \shortform{2'4'5'6'} & 1 &  \shortform{1356} & \shortform{2'4'7'8'} & 1 \\ 
\shortform{1356} & \shortform{3'4'5'7'} & 1 &  \shortform{1356} & \shortform{3'4'6'8'} & 1 &  \shortform{1357} & \shortform{1'3'5'7'} & -2 &  \shortform{1357} & \shortform{2'4'6'8'} & -2 &  \shortform{1358} & \shortform{1'3'5'8'} & -1 \\ 
\shortform{1358} & \shortform{1'3'6'7'} & -1 &  \shortform{1358} & \shortform{1'4'5'7'} & -1 &  \shortform{1358} & \shortform{1'4'6'8'} & -1 &  \shortform{1358} & \shortform{2'3'5'7'} & 1 &  \shortform{1358} & \shortform{2'3'6'8'} & 1 \\ 
\shortform{1358} & \shortform{2'4'5'8'} & 1 &  \shortform{1358} & \shortform{2'4'6'7'} & 1 &  \shortform{1367} & \shortform{1'3'5'8'} & -1 &  \shortform{1367} & \shortform{1'3'6'7'} & -1 &  \shortform{1367} & \shortform{1'4'5'7'} & 1 \\ 
\shortform{1367} & \shortform{1'4'6'8'} & 1 &  \shortform{1367} & \shortform{2'3'5'7'} & -1 &  \shortform{1367} & \shortform{2'3'6'8'} & -1 &  \shortform{1367} & \shortform{2'4'5'8'} & 1 &  \shortform{1367} & \shortform{2'4'6'7'} & 1 \\ 
\shortform{1368} & \shortform{1'3'6'8'} & -2 &  \shortform{1368} & \shortform{2'4'5'7'} & -2 &  \shortform{1378} & \shortform{1'2'5'7'} & 1 &  \shortform{1378} & \shortform{1'2'6'8'} & 1 &  \shortform{1378} & \shortform{1'3'5'6'} & -1 \\ 
\shortform{1378} & \shortform{1'3'7'8'} & -1 &  \shortform{1378} & \shortform{2'4'5'6'} & 1 &  \shortform{1378} & \shortform{2'4'7'8'} & 1 &  \shortform{1378} & \shortform{3'4'5'7'} & -1 &  \shortform{1378} & \shortform{3'4'6'8'} & -1 \\ 
\shortform{13} & \shortform{1'2'3'4'5'7'} & 2 &  \shortform{13} & \shortform{1'2'3'4'6'8'} & 2 &  \shortform{13} & \shortform{1'3'5'6'7'8'} & -2 &  \shortform{13} & \shortform{2'4'5'6'7'8'} & 2 &  \shortform{145678} & \shortform{1'4'} & -2 \\ 
\shortform{145678} & \shortform{2'3'} & -2 &  \shortform{145678} & \shortform{5'8'} & 2 &  \shortform{145678} & \shortform{6'7'} & -2 &  \shortform{1456} & \shortform{1'2'5'8'} & -1 &  \shortform{1456} & \shortform{1'2'6'7'} & 1 \\ 
\shortform{1456} & \shortform{1'4'5'6'} & -1 &  \shortform{1456} & \shortform{1'4'7'8'} & -1 &  \shortform{1456} & \shortform{2'3'5'6'} & -1 &  \shortform{1456} & \shortform{2'3'7'8'} & -1 &  \shortform{1456} & \shortform{3'4'5'8'} & 1 \\ 
\shortform{1456} & \shortform{3'4'6'7'} & -1 &  \shortform{1457} & \shortform{1'3'5'8'} & -1 &  \shortform{1457} & \shortform{1'3'6'7'} & 1 &  \shortform{1457} & \shortform{1'4'5'7'} & -1 &  \shortform{1457} & \shortform{1'4'6'8'} & 1 \\ 
\shortform{1457} & \shortform{2'3'5'7'} & -1 &  \shortform{1457} & \shortform{2'3'6'8'} & 1 &  \shortform{1457} & \shortform{2'4'5'8'} & -1 &  \shortform{1457} & \shortform{2'4'6'7'} & 1 &  \shortform{1458} & \shortform{1'4'5'8'} & -2 \\ 
\shortform{1458} & \shortform{2'3'6'7'} & -2 &  \shortform{1467} & \shortform{1'4'6'7'} & -2 &  \shortform{1467} & \shortform{2'3'5'8'} & -2 &  \shortform{1468} & \shortform{1'3'5'8'} & -1 &  \shortform{1468} & \shortform{1'3'6'7'} & 1 \\ 
\shortform{1468} & \shortform{1'4'5'7'} & 1 &  \shortform{1468} & \shortform{1'4'6'8'} & -1 &  \shortform{1468} & \shortform{2'3'5'7'} & 1 &  \shortform{1468} & \shortform{2'3'6'8'} & -1 &  \shortform{1468} & \shortform{2'4'5'8'} & -1 \\ 
\shortform{1468} & \shortform{2'4'6'7'} & 1 &  \shortform{1478} & \shortform{1'2'5'8'} & 1 &  \shortform{1478} & \shortform{1'2'6'7'} & -1 &  \shortform{1478} & \shortform{1'4'5'6'} & -1 &  \shortform{1478} & \shortform{1'4'7'8'} & -1 \\ 
\shortform{1478} & \shortform{2'3'5'6'} & -1 &  \shortform{1478} & \shortform{2'3'7'8'} & -1 &  \shortform{1478} & \shortform{3'4'5'8'} & -1 &  \shortform{1478} & \shortform{3'4'6'7'} & 1 &  \shortform{14} & \shortform{1'2'3'4'5'8'} & 2 \\ 
\shortform{14} & \shortform{1'2'3'4'6'7'} & -2 &  \shortform{14} & \shortform{1'4'5'6'7'8'} & -2 &  \shortform{14} & \shortform{2'3'5'6'7'8'} & -2 &  \shortform{1567} & \shortform{1'2'3'5'} & -1 &  \shortform{1567} & \shortform{1'2'4'6'} & 1 \\ 
\shortform{1567} & \shortform{1'3'4'7'} & 1 &  \shortform{1567} & \shortform{1'5'6'7'} & -1 &  \shortform{1567} & \shortform{2'3'4'8'} & 1 &  \shortform{1567} & \shortform{2'5'6'8'} & -1 &  \shortform{1567} & \shortform{3'5'7'8'} & -1 \\ 
\shortform{1567} & \shortform{4'6'7'8'} & 1 &  \shortform{1568} & \shortform{1'2'3'6'} & -1 &  \shortform{1568} & \shortform{1'2'4'5'} & -1 &  \shortform{1568} & \shortform{1'3'4'8'} & 1 &  \shortform{1568} & \shortform{1'5'6'8'} & -1 \\ 
\shortform{1568} & \shortform{2'3'4'7'} & -1 &  \shortform{1568} & \shortform{2'5'6'7'} & 1 &  \shortform{1568} & \shortform{3'6'7'8'} & -1 &  \shortform{1568} & \shortform{4'5'7'8'} & -1 &  \shortform{1578} & \shortform{1'2'3'7'} & -1 \\ 
\shortform{1578} & \shortform{1'2'4'8'} & -1 &  \shortform{1578} & \shortform{1'3'4'5'} & -1 &  \shortform{1578} & \shortform{1'5'7'8'} & -1 &  \shortform{1578} & \shortform{2'3'4'6'} & 1 &  \shortform{1578} & \shortform{2'6'7'8'} & 1 \\ 
\shortform{1578} & \shortform{3'5'6'7'} & 1 &  \shortform{1578} & \shortform{4'5'6'8'} & 1 &  \shortform{15} & \shortform{1'2'3'5'6'7'} & -2 &  \shortform{15} & \shortform{1'2'4'5'6'8'} & -2 &  \shortform{15} & \shortform{1'3'4'5'7'8'} & -2 \\ 
\shortform{15} & \shortform{2'3'4'6'7'8'} & 2 &  \shortform{1678} & \shortform{1'2'3'8'} & -1 &  \shortform{1678} & \shortform{1'2'4'7'} & 1 &  \shortform{1678} & \shortform{1'3'4'6'} & -1 &  \shortform{1678} & \shortform{1'6'7'8'} & -1 \\ 
\shortform{1678} & \shortform{2'3'4'5'} & -1 &  \shortform{1678} & \shortform{2'5'7'8'} & -1 &  \shortform{1678} & \shortform{3'5'6'8'} & 1 &  \shortform{1678} & \shortform{4'5'6'7'} & -1 &  \shortform{16} & \shortform{1'2'3'5'6'8'} & -2 \\ 
\shortform{16} & \shortform{1'2'4'5'6'7'} & 2 &  \shortform{16} & \shortform{1'3'4'6'7'8'} & -2 &  \shortform{16} & \shortform{2'3'4'5'7'8'} & -2 &  \shortform{17} & \shortform{1'2'3'5'7'8'} & -2 &  \shortform{17} & \shortform{1'2'4'6'7'8'} & 2 \\ 
\shortform{17} & \shortform{1'3'4'5'6'7'} & 2 &  \shortform{17} & \shortform{2'3'4'5'6'8'} & 2 &  \shortform{18} & \shortform{1'2'3'6'7'8'} & -2 &  \shortform{18} & \shortform{1'2'4'5'7'8'} & -2 &  \shortform{18} & \shortform{1'3'4'5'6'8'} & 2 \\ 
\shortform{18} & \shortform{2'3'4'5'6'7'} & -2 & & & & & & & & & & & & \\ \hline
\end{tabular}
}
\caption[Non-zero terms of the $8$-form $\spinform{9}$ in $\RR^{16}$]{The non-zero terms of the $8$-form $\spinform{9}$ are listed. Reading example: a table entry $||\shortform{145678}\quad\shortform{2'3'}\quad -2||$ means that $\spinform{9}= \dots -2\shortform{1456782'3'}+\dots$.
There are $702$ non-zero coefficients of $\spinform{9}$, and this table exposes $351$ of them, so that 
\[
\spinform{9}=\text{table}+\star\text{table} \qquad
\]
where $\star$ denotes the Hodge star.}\label{explicitspin9}
\end{table}

\begin{re}
The monomials of $\spinform{9}$ can be partitioned in eight different families. For any subset $\{a,b,c,d\}$ of indexes in $\{1,\dots,8\}$, we say that $\shortform{abcd}$ is of \emph{Cayley type} if and only if $\shortform{d}=[\pm]\shortform{a}\times\shortform{b}\times\shortform{c}$ in the double cross product of $\OO\cong\RR^8$ defined by Formula~\eqref{eq:doublecross}. Observe that this definition does not depend on the ordering of $a,b,c,d$.

Then in $\text{table}+\star$ we can recognize the following patterns:
\begin{enumerate}
\item $2$ monomials $\shortform{12345678}$ and $\shortform{1'2'3'4'5'6'7'8'}$, both with coefficient $-14$;
\item $70$ monomials $\shortform{abcda'b'c'd'}$, one for each of the $\binom{8}{4}$ choices $\{a,b,c,d\}\subset\{1,2,\dots,8\}$. Among them, the $14$ of Cayley type have coefficient $\pm 2$, the remaining $56$ have coefficient $\pm 1$, depending on the orientation; 
\item $70$ monomials $\shortform{abcd\alpha'\beta'\gamma'\delta'}$, where $\alpha,\beta,\gamma,\delta$ are all different from $a,b,c,d$. Again, the $14$ of Cayley type have coefficient $\pm 2$, and the remaining $56$ have coefficient $\pm 1$;
\item $336$ monomials $\shortform{abcd\alpha'\beta'\gamma'\delta'}$ with two coincidences, i.e.\ exactly two between $\alpha,\beta,\gamma,\delta$ coincide with two between $a,b,c,d$. Cayley type is here excluded, so that there are $56=70-14$ choices for $\{a,b,c,d\}$, and for each of them there are exactly $6=\binom{4}{2}$ choices for $\{\alpha,\beta,\gamma,\delta\}$: in fact, for each choice of coincidence (for instance, $\alpha=c,\beta=d$), the remaining two indexes are obtained as double cross products (in our example, $\shortform{\gamma}=\shortform{\alpha}\times\shortform{\beta}\times\shortform{a}$ and $\shortform{\delta}=\shortform{\alpha}\times\shortform{\beta}\times\shortform{b}$). Here all monomials have coefficients $\pm 1$, according to the orientation;
\item $28$ monomials $\shortform{abcd\alpha\beta\gamma'\delta'}$, where $\gamma,\delta$ are all different from $a,b,c,d,\alpha,\beta$. The coefficients are $\pm 2$, according to the orientation;
\item $28$ monomials $\shortform{abc'd'\alpha'\beta'\gamma'\delta'}$, where $a,b$ are all different from $c,d,\alpha,\beta,\gamma,\delta$. The coefficients are $\pm 2$, according to the orientation;
\item $84$ monomials $\shortform{abcd\alpha\beta\gamma'\delta'}$, where $\{\gamma,\delta\}\subset\{a,b,c,d,\alpha,\beta\}$. The coefficients are $\pm 2$, and only choices such that the remaining indexes $\{a,b,c,d,\alpha,\beta\}-\{\gamma,\delta\}$ correspond to Cayley type are admitted;
\item $84$ monomials $\shortform{abc'd'\alpha'\beta'\gamma'\delta'}$, where $\{a,b\}\subset\{c,d,\alpha,\beta,\gamma,\delta\}$. The coefficients are $\pm 2$, and only choices such that the remaining indexes $\{c,d,\alpha,\beta,\gamma,\delta\}-\{a,b\}$ correspond to Cayley type are admitted.\hfill\qed
\end{enumerate}
\end{re}

\section{The main formula and its corollaries}\label{sec:charpoly}


Let $M^{16}$ be a Riemannian manifold equipped with a $\Spin{9}$-structure, as in Definition~\ref{def:spin9structure}. The linear algebra developed in Section~\ref{sec:spin9} gives then local K\"ahler matrices on $M$, namely the skew-symmetric matrices $\psi\ug(\psi_{\alpha\beta})$, where $\psi_{\alpha\beta}$ are the K\"ahler forms of the $36$ local almost complex structures $\J_{\alpha\beta}$, for $1\leq\alpha<\beta\leq 9$ (cf.\ Formulas~\eqref{eq:J1} and~\eqref{eq:J2}). Moreover, we have a $\Spin{9}$-form on $M$, that is, the $8$-form locally written as $\spinform{9}$ given by Formula~\eqref{eq:implicitspin9}. We denote it by the same symbol $\spinform{9}$.

A local K\"ahler matrix $\psi$ is a local $2$-form taking values in $\lieso{9}$, and $\psi$, $\psi'$ associated with different local orthonormal bases of sections are related as usual by
\begin{equation}
\psi'= A^{-1}\psi A\enspace,
\end{equation}
where $A$ denotes the change of basis, with values in $\SO{9}$. Thus the characteristic polynomial $\det(tI-\psi)$ is globally defined.

The following is the main theorem of this paper.

\begin{te}\label{teo:main}
The $8$-form $\spinform{9}$ associated with the $\Spin{9}$-structure $V^9\rightarrow M^{16}$ coincides, up to a constant, with the coefficient $\tau_4(\psi)$ of $t^5$ in the characteristic polynomial
\[
\det(tI-\psi)=t^9+\tau_4(\psi)t^5+\tau_8(\psi)t\enspace,
\]
where $\psi$ is any local K\"ahler matrix of $M$. The proportionality factor is given by
\[
360\spinform{9}=\tau_4(\psi)\enspace.
\]
\end{te}

\begin{proof}
The fact that in the characteristic polynomial only the terms of degree $9$, $5$ and $1$ survive was already observed in Proposition~\ref{pr:charpoly}. The $8$-form $\tau_4(\psi)$ is naturally $\Spin{9}$-invariant and thus, if not zero, it has to be proportional to $\spinform{9}$. Then, to compute the proportionality factor, it is sufficient to look at any of the terms of $\spinform{9}$ and $\tau_4(\psi)$. We consider the term $\shortform{12345678}$.

From Table~\ref{explicitspin9}, we see that the coefficient for $\spinform{9}$ is $-14$. As for $\tau_4(\psi)$, we first observe that it can be computed with a summation over the squared Pfaffians of the principal $4\times 4$ submatrices of $\spinform{9}$:
\[
\tau_4(\psi)=\sum_{1\leq \alpha_1 < \alpha_2 <  \alpha_3 <  \alpha_4 \leq 9} ( \psi_{\alpha_1 \alpha_2} \wedge \psi_{\alpha_3 \alpha_4} - \psi_{\alpha_1 \alpha_3} \wedge \psi_{\alpha_2  \alpha_4} + \psi_{\alpha_1 \alpha_4} \wedge \psi_{\alpha_2  \alpha_3} )^2\enspace,
\]
and then we compute it using Formulas~\eqref{28} and~\eqref{8}, thus obtaining $-5040=-14\cdot 360$.
\end{proof}

In particular the theorem stated in the Introduction follows.


Another consequence of Theorem~\ref{teo:main} is:

\begin{co}
The K\"ahler forms of the $\Spin{9}$-structure of $\OO^2$ allow to compute the integral~\eqref{ber} as
\[
\int_{\OP{1}}p_l^*\nu_l\,dl=\frac{\pi^4}{110880\cdot 360}\tau_4(\psi).
\]
\end{co}


When $\Spin{9}$ is the holonomy group of the Riemannian manifold $M^{16}$, the Levi-Civita connection $\nabla$ preserves the vector bundle $V^9$, and the local sections $\I_1,\dots,\I_9$ of $V^9$ induce the K\"ahler forms $\psi_{\alpha\beta}$ on $M$ as local curvature forms.

\begin{co}\label{hol} 
Let $M^{16}$ be a compact Riemannian manifold with holonomy $\Spin{9}$, i.e.\ $M^{16}$ is either isometric to the Cayley projective plane $\OP{2}$ or to any compact quotient of the Cayley hyperbolic plane $\OH{2}$. Then its Pontrjagin classes are given by
\[
p_1(M)=0\enspace,\quad p_2(M)=-\frac{45}{2\pi^4}[\spinform{9}]\enspace,\quad p_3(M)=0\enspace,\quad p_4(M)=-\frac{13}{256\pi^8}[\tau_8(\psi)]\enspace.
\]
\end{co}

\begin{proof}
By Chern-Weil theory the Pontrjagin classes of the vector bundle $V^9\rightarrow M$ are
\[
p_1(V)=0\enspace,\quad 16\pi^4 p_2(V)=\tau_4(\psi)=360[\spinform{9}]\enspace,\quad p_3(V)=0\enspace,\quad 256\pi^8 p_4(V)=[\tau_8(\psi)]\enspace.
\]
On the other hand, for any compact manifold $M$ equipped with a $\Spin{9}$-structure, the following relations hold between the Pontrjagin classes of $V=V^9$ and the Pontrjagin classes of $M$, see~\cite[Page 138]{fr}:
\begin{equation}
\begin{split}
p_1(M)&=2p_1(V)\enspace, \\
p_2(M)&=\frac{7}{4}p_1^2(V)-p_2(V)\enspace, \\
p_3(M)&=\frac{1}{8}\left(7p_1^3(V)-12p_1(V)p_2(V)+16p_3(V)\right)\enspace, \\
p_4(M)&=\frac{1}{128}\left(35p_1^4(V)-120p_1^2(V)p_2(V)+400p_1(V)p_3(V)-1664p_4(V)\right)\enspace.
\end{split}
\end{equation}
Thus, under our hypotheses, from $\tau_2(\psi)=\tau_6(\psi)=0$ we get $p_1(V)=p_3(V)=0$, so that $p_1(M)=p_3(M)=0$, $p_2(M)=-p_2(V)$ and $p_4(M)=-13p_4(V)$. The conclusion follows.
\end{proof}

The Pontrjagin classes of $\OP{2}$ are known for a long time, see~\cite[Page 535]{bh}: $p_2(\OP{2})=6u$ and $p_4(\OP{2})=39u^2$, where $u$ is the canonical generator of $H^8(\OP{2};\ZZ)$. Thus Corollary~\ref{hol} give the following representative forms of the cohomology classes $u$ and $u^2$:
\[
u=[-\frac{15}{4\pi^4}\spinform{9}]=[-\frac{1}{96\pi^4}\tau_4(\psi)]\enspace,\qquad u^2=[-\frac{1}{768\pi^8}\tau_8(\psi)]\enspace.
\]

The volume of $\OP{2}$ with respect to the canonical metric is known to be $6\pi^8/11!$, and the volume of its totally geodesic $\OP{1}\subset\OP{2}$ is the same as the volume of $S^8(\frac{1}{2})$, i.e.\ $\pi^4/840$, cf.~\cite[Page 8]{be}. Thus:
\begin{co}\label{hol'}
On the Cayley projective plane $\OP{2}$ the following relation holds:
\[
[\tau_4(\psi)]^2=12[\tau_8(\psi)]\enspace.
\]
Moreover, the integrals of $\spinform{9}^2$ and $\spinform{9}$ on $\OP{2}$ and on its totally geodesic submanifold $\OP{1}$
give
\[
\int_{\OP{1}}\spinform{9}=-224\vol(\OP{1})\enspace,\qquad \int_{\OP{2}}\spinform{9}^2=-473088\vol(\OP{2})\enspace.
\]
\end{co}



\end{document}